\documentclass{article}[10pt]

\usepackage{soul}
\usepackage{latexsym}
\usepackage{amssymb}
\usepackage{amsmath,amsfonts}
\usepackage{graphicx}
\usepackage{algorithm2e}
\usepackage{algorithmic}
\usepackage{todonotes}
\usepackage{amsthm}
\usepackage{geometry}
\usepackage{algorithm2e}
\usepackage{multirow}
\usepackage{hyperref}
\usepackage{subcaption}
\usepackage{float}
\restylefloat{table}

\numberwithin{equation}{subsection}

\newtheorem{theorem}{Theorem}[section]
\newtheorem{prop}[theorem]{Proposition}
\newtheorem{lemma}{Lemma}[section]
\newtheorem{corollary}{Corollary}[section]
\newtheorem{example}{Example}[subsection]
\newtheorem{definition}{Definition}[section]

\newtheorem{conjecture}[theorem]{Conjecture}

\newtheorem{question}[theorem]{Question}

\newcommand{\dxs}{\Delta x^2}

\newcommand{\na}{\nabla^n_j}

\def\ex{\begin{example}}
\def\eex{\end{example}}
\def\exx{\end{example}}
\def\T{\begin{theorem}}
\def\TT{\end{theorem}}
\def\D{\begin{definition}}
\def\DD{\end{definition}}
\def\l{\begin{lemma}}
\def\ll{\end{lemma}}
\def\c{\begin{corollary}}
\def\cc{\end{corollary}}
\def\cj{\begin{conjecture}}
\def\cjj{\end{conjecture}}
\def\e{\begin{equation}}
\def\ee{\end{equation}}
\def\p{\begin{prop}}
\def\pp{\end{prop}}
\def\q{\begin{question}}
\def\qq{\end{question}}

\def \N {\mathbb{N}}
\def \R {\mathbb{R}}

\def \e {\epsilon}

\def \b {\beta}

\def \O {\Omega}
\def \a {\alpha}

\begin{document}

\baselineskip 15pt

\title{\bf Numerical Analysis of the 1-D Parabolic Optimal Transport Problem}
\author{Abby Brauer\\Lewis and Clark College\\abrauer@lclark.edu \\ \\ Megan Krawick\\Youngstown State University\\mekrawick@student.ysu.edu\\ \\ Manuel Santana\\ Utah State University\\ manuelarturosantana@gmail.com\\ \\ Advisors:\\ Farhan Abedin \\ abedinf1@msu.edu\\ Michigan State University \\ \\ Jun Kitagawa\\kitagawa@math.msu.edu\\ Michigan State University}
\date{}

\maketitle

\begin{abstract}
Numerical methods for the optimal transport problem is an active area of research. Recent work of Kitagawa and Abedin shows that the solution of a time-dependent equation converges exponentially fast as time goes to infinity to the solution of the optimal transport problem. This suggests a fast numerical algorithm for computing optimal maps; we investigate such an algorithm here in the 1-dimensional case. Specifically, we use a finite-difference scheme to solve the time-dependent optimal transport problem and carry out an error analysis of the scheme. A collection of numerical examples is also presented and discussed.

\end{abstract}

\section{Introduction}

\subsection{The Optimal Transport Problem}

The centuries old optimal transport problem asks how to find the cheapest way to transport materials from a given source to a target location \cite{Monge}. In the 1-dimensional case and for the quadratic cost function, the mathematical formulation of the problem is as follows. Let $[A,B],[C,D] \subset \R$ be bounded intervals, representing, respectively, the source and target domains. Consider two positive functions $f: [A,B] \to \R$ and $g: [C,D] \to \R$ satisfying the condition
\begin{equation}\label{massbalance}
\int_A^B f(x) = \int_C^D g(x) = 1.
\end{equation}
The function $f$ can be thought of as describing the mass distribution of a pile of dirt while $g$ describes the depth of the hole the dirt is intended to fill. The mass balance condition \eqref{massbalance} encodes the fact that the amount of dirt is equal to the size of the hole. 

Define the class of admissible transport maps
$$\mathbb{M} = \left\{S : [A,B] \to [C,D] \text{ satisfying } \int_{S^{-1}(E)} f(x) \ dx = \int_{E} g(y) \ dy \text{ for any open set } E \subset [C,D] \right\}.$$
For any $S \in \mathbb{M}$, define the total cost of $S$ to be the quantity
$$\mathcal{C}(S) = \int_A^B |x - S(x)|^2 \ dx.$$
The optimal transport problem is to find a map $T \in \mathbb{M}$ that minimizes the total cost among all maps $S \in \mathbb{M}$, i.e. 
$$\mathcal{C}(T) = \min_{S \in \mathbb{M}} \mathcal{C}(S).$$
By a celebrated result of Brenier \cite{Brenier}, under appropriate conditions on $f$ and $g$, the optimal map $T$ exists and is unique.
In addition, $T(x) = u'(x)$ where $u:[A,B] \to \R$ is a convex function that satisfies the boundary-value problem 
\begin{equation}\label{OT}\tag{O-T}
 \begin{cases}
 u''(x) = \frac{f(x)}{g(u'(x))},\\
 u'(A) = C, \; u'(B) = D.
 \end{cases}   
\end{equation}
Notice that \eqref{OT} implies
$$\int_A^x f(p) \ dp = \int_A^x g(u'(p)) u''(p) \ dp = \int_{u'(A)}^{u'(x)} g(q) \ dq =  \int_{C}^{u'(x)} g(q) \ dq.$$
    If we define the cumulative distribution functions of $f$ and $g$, respectively, as
    \begin{equation}\label{CDF}
    F(x) := \int_A^x f(p) \ dp, \quad x \in [A,B],\qquad 
    G(y) := \int_C^y g(q) \ dq, \quad y \in [C,D],
    \end{equation}
    we then have the relation
    $$F(x) = G(u'(x)) \quad \text{for all } x \in [A,B].$$
    Since $g$ is positive on $[C,D]$, we have $G'(y) = g(y) > 0$, so $G$ is strictly increasing, hence invertible. Therefore, the optimal map $T$ can be expressed in terms of $F$ and $G$ as
    \begin{equation}\label{quantileInverse}
    T(x) = u'(x) = G^{-1}(F(x)) \quad \text{for all } x \in [A,B].
    \end{equation}
In practice, given $f$ and $g$, it is difficult to compute $F$ and $G$ analytically. This provides motivation to develop alternate numerical methods of obtaining the optimal map $T$. Much work has been done on the numerical approximation of optimal maps in low dimensions \cite{Benamou, Benamou2, 1901.05108}. Here we consider an approach based on a  time-dependent version of \eqref{OT} studied in \cite{AbedinKitagawa, Kitagawa12} and referred to as the \emph{parabolic optimal transport problem}. In our setting, this problem can be stated as follows: find a time-dependent function $v(t,x)$ that satisfies
\begin{equation}\label{POT}\tag{Parabolic O-T}
\begin{cases}
v_t = \log(v_{xx}) - \log\left(\frac{f(x)}{g(v_x)}\right ) & \text{in } (0,\infty) \times (A,B),\\
v_x(t,A) = C, \quad v_x(t,B) = D & \text{for all } t \geq 0,\\
v(0,x) = u_0(x),\\ 
v(t, \cdot) \text{ strictly convex for all } t \geq 0.
\end{cases}
\end{equation} 
Here, $u_0(x)$ is a given convex function on $[A,B]$ that satisfies $u_0'(A) = C$ and $u_0'(B) = D$. It is shown in \cite{Kitagawa12} that $\lim_{t \to \infty} v(t,x) = u(x)$ where $u$ solves \eqref{OT}. The more recent work \cite{AbedinKitagawa} shows the convergence is exponentially fast in $t$.

\subsection{Discretization of the Problem }
    The purpose of this paper is to carry out a numerical approximation of \eqref{POT} and study a number of examples. To numerically approximate \eqref{POT}, we choose to use a finite difference scheme. This requires discretizing the interval $[A,B]$ using $J \in \mathbb{N}$ grid points
     $$x_j = A + \frac{j(B-A)}{J}, \quad j = -1,\ldots, J+1.$$
     The range of indices is chosen this way in order to provide an extra point outside each end of the interval $[A,B]$. We will use the notation
     $$\Delta x = \frac{B-A}{J}$$
     to denote the spatial grid resolution and use the short-hand $\Delta x^2 := (\Delta x)^2$. To discretize the time interval $[0,\infty)$, we let $\{t_n\}_{n = 0}^{\infty}$ be a non-negative sequence of strictly increasing time values with $t_0 = 0$. Denote the $n$-th time step by
     $$\Delta t_n = t_{n+1} - t_n.$$
     We denote by $\mathcal{G}$ the set of all grid points $\{(t_n, x_j):n \in \N, \  j \in -1, \ldots, J+1 \}$. 
     \newline
     In order to motivate the finite difference scheme, we recall the following consequences of the Taylor Remainder Theorem:
     \begin{align}
     v_x(t_n,x_j) &= \frac{v(t_n, x_{j+1}) - v(t_n,x_{j-1})}{2\Delta x} - \frac{v_{xxx}(t_n,x_j+\psi) + v_{xxx}(t_n,x_j-\psi) }{12}\dxs \quad \text{for some } \psi \in (0,\Delta x),\label{eqn: vx taylor}\\
    v_{xx}(t_n,x_j) &= \frac{v(t_n, x_{j+1}) + v(t_n,x_{j-1}) - 2v(t_n,x_j)}{\Delta x^2} - \frac{v_{xxxx}(t_n,x_j+\gamma) + v_{xxxx}(t_n,x_j-\gamma)}{24}\dxs \quad \text{for some } \gamma \in (0,\Delta x),\label{eqn: vxx taylor}\\
    v_t(t_n,x_j) & = \frac{v(t_{n+1},x_j) - v(t_n,x_j)}{\Delta t_n} - \frac{v_{tt}(t_n+\kappa,x_j)}{2}\Delta t_n \quad \text{for some } \kappa \in (0,\Delta t_n).\label{eqn: vt taylor}
\end{align}

Our goal is to construct a grid function $U: \mathcal{G} \to \R$ such that $U(t_n,x_j) \approx v(t_n,x_j)$ for $j = 0, \ldots J$ and $n \in \N$, where $v$ is the solution of \eqref{POT}. We will, from here onward, use the short-hand $v^n_j = v(t_n,x_j)$ and $U^n_j = U(t_n,x_j)$.
Neglecting the terms of order $\Delta t_n$ and $\Delta x^2$ in the Taylor expansions of $v$ above, we obtain the definition of the first and second order finite difference operators acting on the approximation $U$. For convenience we will define two operators for some arbitrary function $\phi$.
\D The first order centered difference operator $\nabla^n_j$ and the second order centered difference operator $\Delta^n_j$ are defined as
\begin{align*}
\nabla^n_j \phi & := \frac{\phi^n_{j+1} - \phi^n_{j-1}}{2\Delta x}, \quad j = 0, \ldots J,\\
\end{align*}
\begin{align*}
\Delta^n_j \phi & := \frac{\phi^n_{j+1} + \phi^n_{j-1} - 2\phi^n_j}{\Delta x^2} \quad j = 0, \ldots J.
\end{align*} 
\DD

The approximation for $v_{xx}$ at the boundary points $x = A, B$ requires using the boundary conditions in \eqref{POT}. We use a backward difference first space derivative approximation for the boundary at $A$ and a forward difference approximation for the first space derivative at $B$ with the exact values for these derivatives as given by the Neumann boundary conditions, 
\begin{align}\label{Cerror}
    C = v_x(t_n,A) &= \frac{v(t_n, x_1) - v(t_n,x_{-1})}{2\Delta x} - \frac{v_{xxx}(t_n, A + \psi) + v_{xxx}(t_n, A - \psi)}{12} \Delta x^2 \quad \text{for some } \psi \in (0,\Delta x)\\
    D = v_x(t_n,B) &= \frac{v(t_n,x_{J+1})-v(t_n, x_{J-1})}{2 \Delta x} - \frac{v_{xxx}(t_n,B + \psi) + v_{xxx}(t_n,B - \psi)}{12} \Delta x^2 \quad \text{for some } \psi \in (0,\Delta x).
\end{align}

Utilizing the definition of the finite difference method presented in \cite{Tadmor2012}, we implement the following centered difference approximation of \eqref{POT}
\begin{equation}\label{finitedifference}\tag{F-D}
\begin{cases}
     U_j^{n+1} = \Bigg(\log\left(\Delta^n_j U \right)-\log\left(\frac{f(x_j)}{g\left(\nabla^n_j U\right)}\right)\Bigg)\Delta t_n + U^n_j, \quad j = 0, \ldots J,\\
     \nabla U^n_0 = C, \quad \nabla U^n_J = D,\\
     U^n_{-1} := U^n_1 - 2C\Delta x, \quad U^n_{J+1} = U^n_{J-1} + 2D \Delta x,\\
     U^0_j = u_0(x_j).
\end{cases}
\end{equation}

\subsection{Structure of the Paper}
The remainder of this paper is structured as follows. In Section \ref{sec: errorAnalysis} we show that the error between our numerical approximation and the true solution of \eqref{POT} is bounded by quantities depending on previous time-steps. In Section \ref{sec:code} we show how to measure the asymptotic closeness of \eqref{finitedifference} to \eqref{OT}. Finally in Section \ref{sec: examples} we discuss the code for implementation of the finite difference scheme as well as numerical findings and applications to quantile functions. The proofs of explicit derivative bounds for the solution of \eqref{POT} used in calculations are given in the appendix.

\section{Error Analysis of Finite Difference Scheme}\label{sec: errorAnalysis}
In order for a numerical approximation to be effective the error at a given time step must be bounded by quantities known from the previous steps. In this section we establish such error bounds for \eqref{finitedifference}. Recall $\Delta ^n_j U$ is a finite approximation of the second derivative and therefore has some error. Although we would expect $\Delta^n_j U$ to stay positive because it is approximating a convex function, how to guarantee this is not yet clear. Therefore we must assume the condition of $\Delta^n_j U$ staying positive for all $n$ in the following error analysis.\\ 
We first define the infinity norm $||\cdot^n||_\infty$ to be the maximum value of $\cdot$ at time step $n$ for all points in $\mathcal{G}$. We will also now define several bounds on the derivatives of $v$.

\D Define the derivative bounds $K$, $\Gamma$, $\Psi$ $\delta_1$, and $\delta_2$ to satisfy
\begin{align*}
    K &\geq |v_{tt}(t, x)|,\\
    \Gamma &\geq |v_{xxxx}(t_n, x_j)|,\\
    \Psi & \geq |v_{xxx} (t_n, x_j)|,\\
    0 &< \delta_1 \leq v_{xx}(t, x) \leq \delta_2.
\end{align*}
for all $t\in [0, \infty)$ and $x\in [A, B]$.
\DD

It is possible to explicitly calculate the constants $K, \Gamma, \Psi, \delta_1, \delta_2$ in terms of the mass distributions $f, g$ and the initial function $v_0$. A full discussion of these calculations is given in the appendix.

We are now prepared to state our error bound for our numerical scheme.
\T\label{theo: error bound}
Assuming $\Delta^n_j U$ remains positive at every time step $n$, positive time and spatial grid steps, and that the following is true for $\Delta x$ and $\Delta t$:
\begin{align}\label{errorConditions}
    \Delta x &= \min\left\{\frac{3 \delta_1}{2\Psi}, \sqrt{\frac{6\delta_1}{\Gamma}} \right\}, \notag \\
    \frac{\max|g'(y)|\Delta t_n}{2\min g(y)  \Delta x}  &\le \frac{\Delta t_n}{\dxs \min\left\{\frac{1}{2} \delta_1, \min\{\Delta^n_j U\}\right\}} \le \frac{1}{2},
\end{align}
\eqref{finitedifference} has the following maximum error bound on the interior points (points not on the boundary, $j = 1, \dots, J-1$) at time step n.
\begin{equation}\label{espError}
    ||U^n - v^n||_\infty \leq  \sum_{i=0}^{n-1} \left( \Delta t_{i} \left(\frac{\Delta t_{i}}{2}K + \frac{\Delta x^2}{6 \delta_1}\Gamma \right) +  \frac{\max |g '|}{\min g} \left( \frac{\dxs}{6} \Psi  \right)  \right)
\end{equation}
\TT

This theorem shows that our finite difference scheme is close to the real solution of \eqref{POT} for all $t_n$. We will spend the rest of Section \ref{sec: errorAnalysis} proving Theorem \ref{theo: error bound}. For a more formal discussion of the efficacy of finite difference schemes see \cite{Tadmor2012}.

\subsection{Calculation of Local Error}\label{LocalErrorSection}
In this section we prove the first necessary lemma for proving Theorem \ref{theo: error bound}. We start with several definitions.

\D
\begin{enumerate}
\item[(i)] Local approximation: 
$$V^{n+1}_j := \Bigg(\log\left(\Delta^n_j v \right)-\log\left(\frac{f(x_j)}{g\left(\nabla^n_j v\right)}\right)\Bigg)\Delta t_n + v^n_j , \quad  j = 1,\ldots, J-1.$$

\item[(ii)] Local error $\tau$ at grid point $(x_j, t_n)$:
$$\tau^n_j := V^n_j - v_j^n.$$

\item[(iii)] Local discretization error $\theta$ at grid point $(x_j, t_n)$:
\begin{equation}
    \theta^n_j := \frac{v^{n+1}_j - v^n_j}{\Delta t_n} - \log\left(\Delta^n_j v \right) + \log\left(\frac{f(x_j)}{g(\na v)}\right).
\end{equation}
\end{enumerate}
\DD

With these definitions in hand we state the first lemma necessary for proving Theorem \ref{theo: error bound}:
\l \label{tau lemma} Assuming the conditions  \eqref{errorConditions}, the local error $\tau$ for any point on time step $n$ has an upper bound
$$||\tau^{n+1}||_\infty \leq \Delta t_n \left( \frac{\Delta t_n K}{2} + \frac{\Delta x^2 \Gamma }{6 \delta_1}+ \frac{\dxs \max |g '| \Psi }{6\min g}  \right)$$
\ll

\begin{proof}
First we note that local discretization error provides a useful identity
\begin{equation}\label{eqn: tau expression}
    V^{n+1}_j - v^{n+1}_j= \tau^{n+1}_j = -\Delta t_n \theta^n_j.
\end{equation}
Therefore, to calculate the bounds on the local error, $\tau^{n+1}_j$, we first need to estimate $\theta^n_j$. Using \eqref{POT}, we can substitute for $f(x_j)$ and get

\begin{equation}\label{eqn: theta expression}
    \theta^n_j = \frac{v^{n+1}_j - v^n_j}{\Delta t_n}  - v_t(t_n,x_j) - \left(\log\left(\Delta^n_j v\right) - \log (v_{xx}(t_n,x_j)) \right) - (\log \left( g\left( \na v\right)\right) - \log\left( g\left( v_x(t_n, x_j)\right)\right)).
\end{equation}
From  \eqref{eqn: vt taylor}, we have  \begin{equation}\label{timediff}
\frac{v^{n+1}_j - v^n_j}{\Delta t_n}  - v_t(t_n,x_j) = \frac{v_{tt}(t_n+\kappa,x_j)}{2}\Delta t_n .
\end{equation}
To simplify $\log\left(\Delta^n_j v\right) - \log (v_{xx}(t_n,x_j)) $ we use the Mean Value Theorem to find a number $\eta$  between  $v_{xx}(t_n,x_j)$ and $\Delta^n_j v$ such that
$$
    \log\left(\Delta^n_j v\right) - \log({v_{xx}}(t_n,x_j)) = \frac{1}{\eta}(\Delta^n_j v - v_{xx}(t_n,x_j)),
$$
The Taylor approximation \eqref{eqn: vxx taylor} then shows
\begin{equation}\label{eqn: v log diff}
    \log\left(\Delta^n_j v\right) = \log(v_{xx}(t_n, x_j)) + \frac{\dxs}{\eta}\left( \frac{v_{xxxx}(t_n,x_j+\gamma) + v_{xxxx}(t_n,x_j-\gamma)}{24} \right).
\end{equation}
The Mean Value Theorem implies there exists a number $\mu$ between $g(v_x(t_n, x_j))$ and $g\left(\na v\right)$, and a number $\chi$ between $g'(v_x(t_n, x_j))$ and $g'(\na v)$ such that
$$
\log \left( g\left( \na v\right)\right) - \log\left( g\left( v_x(t_n, x_j)\right)\right) = \frac{1}{\mu}\left( g\left(\na v\right) - g(v_x(t_n, x_j)) \right) = \frac{\chi }{\mu} \left( \na v - v_x(t_n, x_j)\right).
$$
Using the Taylor expansion \eqref{eqn: vx taylor}, we find that
\begin{equation}\label{Glogs}
    \log \left( g\left( \na v\right)\right) - \log\left( g\left( v_x(t_n, x_j)\right)\right) =  \frac{\chi \dxs}{\mu} \left(\frac{v_{xxx}(t_n,x_j+\psi) + v_{xxx}(t_n,x_j-\psi) }{12}\right).
\end{equation}
Substituting \eqref{timediff}, \eqref{eqn: v log diff} and \eqref{Glogs} into \eqref{eqn: theta expression} gives us

\begin{small}
\begin{align}\label{eqn: tau expression 1}
    \tau^{n+1}_j = & -\Delta t_n\theta^n_j \notag \\
    = & \Delta t_n \left(-\frac{\Delta t_n}{2}v_{tt}(t_n+\kappa, x_j) + \frac{\Delta x^2}{24 \eta} \left( v_{xxxx}(t_n,x_j+\gamma) + v_{xxxx}(t_n,x_j-\gamma)\right) + \frac{\Delta x^2 \chi}{12 \mu} \left( v_{xxx}(t_n,x_j+\psi) + v_{xxx}(t_n,x_j-\psi)\right)\right).
\end{align}
\end{small}
Using the derivative bounds in \eqref{eqn: tau expression 1} shows
\begin{equation}\label{eqn: tau expression 2}
|\tau^{n+1}_j| \leq \Delta t_n \left(\frac{K \Delta t_n}{2} + \frac{\Delta x^2 \Gamma}{12 \eta} + \frac{\Delta x^2 \Psi \chi}{6 \mu}\right).
\end{equation}
If we now choose $\Delta x$ to satisfy 
\begin{align}\label{eqn: delta x restriction 1}
    \Delta x^2\leq \frac{6\delta_1}{\Gamma},
\end{align}
we obtain the inequality 
$$
\frac{1}{12}\max_{x,t} |v_{xxxx}|  \dxs \le \frac{1}{2}\min_{x,t} v_{xx}.
$$
Then by the Taylor expansion \eqref{eqn: vxx taylor}, we obtain 
\begin{equation*}
    \min_{j} \{\Delta ^n_j v\} \ge \frac{1}{2}\min_{x,t} v_{xx} ,
\end{equation*}
hence
\begin{equation*}
    \eta \ge \min_j\left\{\Delta^n_j v, v_{xx}(t_n, x_j) \right\}\geq \frac{\delta_1}{2}.
\end{equation*}
Additionally we know $\chi$ is bounded from above by $\max |g'|$ and $\mu$ is bounded from below by $\min g$. Substituting the constants $K$, $\Gamma$,$\Psi$, and $\delta_1$ into \eqref{eqn: tau expression 2} and using the triangle inequality with our $g$ bounds we obtain
\begin{equation}
    |\tau^{n+1}_j| \leq \Delta t_n \left( \frac{\Delta t_n K}{2} + \frac{\Delta x^2 \Gamma }{6 \delta_1}+ \frac{\dxs \max |g '| \Psi }{6\min g}  \right) \quad \text{for all } j \in \{1, \ldots, J-1\}.
\end{equation}
under the restriction \eqref{eqn: delta x restriction 1}. Finally we can take the max norm of $\tau^{n+1}_j$ over all $j = 1, \ldots, J-1$ to obtain
\begin{equation}\label{taudef}
    ||\tau^{n+1}||_\infty \leq \Delta t_n \left( \frac{\Delta t_n K}{2} + \frac{\Delta x^2 \Gamma }{6 \delta_1}+ \frac{\dxs \max |g '| \Psi }{6\min g}  \right)
\end{equation}
\end{proof}

\subsection{Calculation of Total Error}\label{totalErrorSection}
In this section we prove an additional lemma. We must first define another error term.
\D Define the error term  $\varepsilon$ 
$$\varepsilon^n_j := U^{n+1}_j - V^{n+1}_j.$$
\DD

\begin{lemma}\label{eps Lemma}
Assuming $\Delta ^n_j U$ stays positive for all $n$ and the conditions in \ref{errorConditions}, the $\varepsilon$ is bounded by:
$$
||\varepsilon^n||_\infty \leq \sum^{n-1}_{i = 0} \Delta t_i \left( \frac{\Delta t_i K}{2} + \frac{\Delta x^2 \Gamma }{6 \delta_1}+ \frac{\dxs \max |g '| \Psi }{6\min g}  \right).
$$
\end{lemma}

\begin{proof}
Note by definition
\begin{equation}\label{e1}
    \varepsilon^n_j  = \left(\log(\Delta^n_j U) - \log(\Delta^n_j v) - \left(\log\left( g\left(\na v\right) \right)  - \log\left( g\left(\na U\right) \right) \right) \right)\Delta t_n + (U^n_j - v^n_j).
\end{equation}
By the Mean Value Theorem, there is some number $\rho$ between $g(\na U)$ and $g\left(\na v\right)$, and another number $\omega$ between $g'(\na U)$ and $g'(\na v)$ such that
\begin{equation}\label{e3}
    \log \left(g\left(\na v \right)\right) - \log \left(g\left( \na U \right)\right) = \frac{\omega}{ 2\Delta x \rho} \left(v^n_{j+1} - U^n_{j+1} + v^n_{j-1} - U^n_{j-1} \right).
\end{equation}
Similarly, there is some numnber $\xi$ between $\Delta^n_j v$ and $\Delta^n_j U$ such that
\begin{equation}\label{e2}
    \log(\Delta^n_jU) - \log(\Delta^n_j v) =  \frac{1}{\xi}(\Delta^n_j U - \Delta^n_j v).
\end{equation}
Substituting \eqref{e2} and \eqref{e3} into \eqref{e1} shows
\begin{equation*}
  \varepsilon^{n+1}_j = \left(\frac{1}{\xi}(\Delta^n_j U - \Delta^n_j V) + \frac{\omega}{ 2\Delta x \rho} \left(v^n_{j+1} - U^n_{j+1} + v^n_{j-1} - U^n_{j-1} \right) \right) \Delta t_n + (U^n_j - V^n_j).
\end{equation*}
To continue we must bound $\xi$ from below by known values. Note that the following inequality holds true.
$$ \xi \ge \min\left\{\min_j\{\Delta^n_j v\},\min_j\{ \Delta^n_j U\}\right\}.$$
If \eqref{eqn: delta x restriction 1} holds, then  $\min_j\{\Delta^n_j v\} \ge \frac{1}{2}\delta_1$. Therefore,
\begin{equation*}
    \frac{1}{\xi} \le \frac{1}{\min\left\{\frac{1}{2}\delta_1,\min\limits_j\{ \Delta^n_j U\}\right\}}.
\end{equation*}
Let us define the quantities
$$r := \frac{\Delta t_n}{ \dxs \min\left\{\frac{1}{2}\delta_1,\min\limits_j\{ \Delta^n_j U\}\right\}}, \qquad s := \frac{\max|g'|\Delta t_n}{2 \min g  \Delta x}.$$
Notice that we can bound $\rho$ from above by $\max|g'|$ and bound $\omega$ from below by $\min g$. Using the definition of the second order difference operator $\Delta^n_j$ and the triangle inequality, we find that
\begin{equation*}
    |\varepsilon^{n+1}_j| \le |r-s||U^n_{j+1} - v^n_{j+1}| + |1-2r||U^n_j - v^n_j|+ |r+s||U^n_{j-1} - v^n_{j-1}|.
\end{equation*}
From the triangle inequality we also know $|U^n_j - v^n_j| \le |\varepsilon^n_j| + |\tau^n_j|$, thus
\begin{equation*}
    |\varepsilon^{n+1}_j| \le |r - s|(|\varepsilon^n_{j+1}| + |\tau^n_{j+1}|) + |1-2r|(|\varepsilon^n_j| + |\tau^n_j|) + |r + s|(|\varepsilon^n_{j-1}| + |\tau^n_{j-1}|).
\end{equation*}
Replacing all quantities of $|\varepsilon^n_j|$ and $|\tau^n_j|$ with their maximum norms
\begin{equation*}
   | \varepsilon^{n+1}_j| \le \left(|r - s| + |1-2r| + |r + s|\right)(||\varepsilon^{n}||_\infty + ||\tau^{n}||_\infty).
\end{equation*}
Now if we choose $\Delta t_n$ and $\Delta x$ such that $s < r \le \frac{1}{2}$, then 
\begin{equation*}\label{s<r}
    \varepsilon^{n+1}_j \le ||\varepsilon^{n}||_\infty + ||\tau^{n}||_\infty.
\end{equation*}
Iterating the inequality over $n$ shows
\begin{equation*}
    ||\varepsilon^{n}||_\infty \leq ||\varepsilon^{0}||_\infty + ||\tau^{n-1}||_\infty + \dots + ||\tau^0||_\infty \quad \text{for all } n \geq 1.
\end{equation*}
Lastly recall that by definition $U^0 - V^0 = 0$ therefore $||\varepsilon^0||_\infty = 0$ and we have
\begin{equation*}\label{tauError}
    ||\varepsilon^{n}||_\infty \leq ||\tau^{n-1}||_\infty + \dots + ||\tau^0||_\infty. 
\end{equation*}
Using the bound for $\tau$ from \eqref{taudef}, we finally get
\begin{equation}
    ||\varepsilon^n||_\infty \leq \sum^{n-1}_{i = 0} \Delta t_i \left( \frac{\Delta t_i K}{2} + \frac{\Delta x^2 \Gamma }{6 \delta_1}+ \frac{\dxs \max |g '| \Psi }{6\min g}  \right).
\end{equation}
\end{proof}

We now have all the tools at hand to prove Theorem \ref{theo: error bound}. 
\begin{proof}[Proof of Theorem \ref{theo: error bound}]

Using the triangle inequality and taking the maximum norm we can say.
\begin{equation*}
    ||U^n_j - v(t_n, x_j)||_\infty \le ||\tau^n_j||_\infty + ||\varepsilon^n_j||_\infty 
\end{equation*}
By Lemmas \ref{tau lemma} and \ref{eps Lemma},
\begin{equation}\label{espErrorAgain}
    ||U^n - v^n||_\infty \leq  \sum_{i=0}^{n-1} \left( \Delta t_{i} \left(\frac{\Delta t_{i}}{2}K+ \frac{\Delta x^2}{6 \delta_1}\Gamma \right) +  \frac{\max |g '|}{\min g} \left( \frac{\dxs}{6} \Psi  \right)  \right).
\end{equation}
\end{proof}


\subsection{Addressing Boundary Conditions}
Due to the Neumann boundary conditions we must analyze error conditions at the boundaries separately. This process is nearly identical as sections \ref{LocalErrorSection} and \ref{totalErrorSection}. Additionally calculations on the left and right boundary are almost identical, and we will just focus on the left boundary point.

We begin with our calculation of the local error. Using the same method of section 3.1, the error term from (\ref{eqn: vt taylor}), and (\ref{Cerror}) gives us
\begin{equation*}
    \tau ^n_0 = \Delta t_n\left( -\frac{\Delta t_n}{2}(v_{tt}(t_n + \kappa, 0)) +  \frac{1}{\lambda}\left(\frac{\Delta x}{3} (v_{xxx}(t_n, \psi) \right) + \frac{\chi }{\mu} \left(\frac{\Delta x^2}{6} v_{xxx}(t_n, x_j + \gamma)\right) \right),
\end{equation*}
with $\lambda$ being some point between $\Delta ^n_0 v$ and $v_{xx}(t_n, 0)$.
We assume $\Delta x$ in the boundary case must also satisfy $\frac{1}{3}|v_{xxx}|\Delta x \le \frac{1}{2}v_{xx}$, which leads to
\begin{equation*}
    \min_{j}\left\{\Delta^n_jv\right\} \ge \frac{1}{2}\min_{x,t}\{v_{xx}\}.
\end{equation*}
This implies that $\lambda < \frac{1}{2} \delta_1$ by similar methods as used before, and results in a bound on $\tau^n_j$ at boundary point $a$ 
\begin{equation*}
    |\tau ^n_0| \leq \Delta t_n \left(\frac{\Delta t_n}{2}K +  \frac{2 \Delta x}{3 (\delta_1)}\Psi + \frac{\max |g '|}{\min g} \left( \frac{\dxs}{6} \Psi  \right) \right).
\end{equation*}
Now calculating our total error at the boundary begins the same way as in section 3.2. If we let $r$ and $s$ be the same quantities then, by use of a similar log approximation as well as the triangle inequality 
\begin{equation*}
    |\varepsilon^n_0| \leq 
    |r + s||U^{n -1}_{1} - V^{n - 1}_{1}|
      + |1-r - s||U^{n - 1}_0 + V^{n - 1}_0|.    
\end{equation*}
We can now replace $|U^n_0 - V^n_0|$ terms with $|\varepsilon^n_0| + |\tau^n_0|$ terms by another triangle inequality
\begin{equation*}
    |\varepsilon^n_0| \leq 
    |r + s|(| \varepsilon^{n - 1}_{1}| + |\tau^{n - 1}_{1}|) +|1 - r - s|(| \varepsilon^{n - 1}_0| +|\tau^n_0|).
\end{equation*}
If $s < r \le \frac{1}{2}$ as was necessary for \eqref{s<r} then it is implied that $ r + s < 1$. Taking maximums over $j$ in $\{0, 1\}$ gives us
\begin{equation*}
    |\varepsilon^n_0| \leq \max_{j}\{\tau^{n - 1}_j\} + ... + \max_{j}\{\tau^{0}_j\}.
\end{equation*}
This allows us to show a final bound on the boundary conditions,
\begin{equation}\label{eqn: boundary}
    |U^n_0 - v^n_{0}| \leq \sum_{i=0}^{n-1} \left(\Delta t_i \left(\frac{\Delta t_i}{2}K +  \frac{\Delta x}{3 \delta_1}\Psi +  \frac{\max |g '|}{\min g} \left( \frac{\dxs}{6} \Psi  \right)  \right)\right).
\end{equation}

\section{Asymptotic Error Analysis}\label{sec:code}
To ensure that our implementation of the code provides an accurate numerical approximation of the optimal map $T(x) = u'(x)$, where $u$ solves \eqref{OT}, we must show that  $\max_j|\na U - u'(x_j)|$ is within a desired tolerance. In this section we show the error between \eqref{OT} and \eqref{finitedifference} is controlled by a quantity that can be calculated at each time step. 

Let $\O \subset [A,B]$ and let $S: \O \to [C,D]$ be an increasing function. Recall the definition of the cumulative distribution functions $F$ and $G$ in \eqref{CDF}. Define the error function of $S$ as
$$E(S,x) := |F(x) - G(S(x))|, \quad x \in \O.$$
Notice that if $T$ is the optimal map between $f$ and $g$, then $E(T,x) = 0$ for all $x \in [A,B]$. Next, we see that for any $x \in \O$,
$$E(S,x) = |F(x) - G(S(x))| = |G(T(x)) - G(S(x))| \geq \left(\min_{[C,D]} g \right)|T(x) - S(x)|.$$
Therefore,
$$\max_{x \in \O} |T(x) - S(x)| \leq \frac{\max_{x \in \O} E(S,x)}{\min_{[C,D]} g}.$$
It follows that for any $\sigma > 0$,
\begin{equation}\label{asymptoticerrorbound}
\max_{x \in \O} E(S,x) \leq \sigma \min_{[C,D]} g \quad \Rightarrow \quad \max_{x \in \O} |T(x) - S(x)| \leq \sigma.
\end{equation}

\T
Let $T(x) = u'(x)$ be the optimal map, where u solves \eqref{OT}. Asumming $\Delta^n_j U$ is strictly positive for all $n$. Given a tolerance $\sigma > 0$.
\begin{equation}
\max_{j=0,\ldots,J} E(\nabla^n_j U,x_j) \leq \sigma \quad \Rightarrow \max_{j=0,\ldots,J} |T(x_j) - \nabla^n_j U| \leq \frac{\sigma}{\min_{[C,D]} g} 
\end{equation}
\TT
\begin{proof}

We apply the estimate \eqref{asymptoticerrorbound} to the finite-difference scheme \eqref{finitedifference}. Let $\O = \{x_0,\ldots, x_J \} \subset [A,B]$ be the set of spatial grid points. For each $n \in \N$, denote the map $S_n : \O \to [C,D]$ as
$$S_n(x_j) := \nabla^n_j U, \quad x_j \in \O.$$
Therefore, we have $S_n(x_0) = C$ and $S_n(x_J) = D$. We check that $S_n$ is an increasing function on $\O$, and hence maps into $[C,D]$. This is a condition of the optimal map as implied by \eqref{OT}. 

Assume $j \in \{1, \ldots J \}$. Then
\begin{align*}
S_n(x_j) - S_n(x_{j-1}) & = \nabla^n_j U - \nabla^n_{j-1} U \\
& = \left(\frac{U^n_{j+1} - U^n_{j-1}}{2\Delta x} \right) -  \left(\frac{U^n_{j} - U^n_{j-2}}{2\Delta x} \right) \\
& = \frac{\Delta x}{2} \left[ \left(\frac{U^n_{j+1} - U^n_{j-1}}{\Delta x^2} \right) -  \left(\frac{U^n_{j} - U^n_{j-2}}{\Delta x^2} \right) \right]\\
& = \frac{\Delta x}{2} \left[ \left(\frac{U^n_{j+1} + U^n_{j-1} - 2U^n_j - U^n_{j-1} + 2U^n_j - U^n_{j-1}}{\Delta x^2} \right) -  \left(\frac{U^n_{j} - U^n_{j-2}}{\Delta x^2} \right) \right]\\
& = \frac{\Delta x}{2} \left[ \left(\frac{U^n_{j+1} + U^n_{j-1} - 2U^n_j}{\Delta x^2} \right) + \left(\frac{ - U^n_{j-1} + 2U^n_j - U^n_{j-1} -U^n_{j} + U^n_{j-2}}{\Delta x^2} \right) \right]\\
& = \frac{\Delta x}{2} \left[ \left(\frac{U^n_{j+1} + U^n_{j-1} - 2U^n_j}{\Delta x^2} \right) + \left(\frac{U^n_{j} + U^n_{j-2} - 2U^n_{j-1}}{\Delta x^2} \right) \right]\\
& =  \frac{\Delta x}{2} \left( \Delta^n_j U + \Delta^n_{j-1} U \right) \geq 0.
\end{align*}
\end{proof}

Notice that the Theorem above requires the condition $\Delta^n_j U \geq 0$ for all $j = 0, \ldots, J$. We note that in practice, if the code does not encounter a domain error at time step $n+1$, then $\Delta^n_j U > 0$ for all $j$. The theorem above also shows that given a tolerance $\sigma > 0$, if there exists some $N(\sigma) \in \N$ such that $$\max_{j = 0,\ldots J} E(\nabla^{N(\sigma)}_j U,x_j) \leq \sigma$$
then
\begin{equation}\label{OTconverge}
    \max_{j=0,\ldots,J} |T(x_j) - \nabla^{N(\sigma)}_j U| \leq \frac{\sigma}{\min_{[C,D]} g}.
\end{equation}
Since the quantity $\max_j E(\nabla^n_j U, x_j)$ can be computed at each time step $n$, we can run our code to the time step $n = N(\sigma)$ for which  $\max_j E(\nabla^n_j U, x_j)$ is less than a specified tolerance $\sigma$. We are not able to guarantee our scheme will always be able to reach such a specified tolerance in a finite number of steps, but if this tolerance is reached, then we can conclude using \eqref{OTconverge} that the map $\nabla^n_j U$ is equal to the optimal map on the grid points $x_1, \ldots, x_J$ up to a quantifiable error.
It should be noted that in practice it is often simpler to use numerical integration to evaluate $E(\nabla^n_j U, x_j)$. Therefore, \eqref{OTconverge} will hold up to the accuracy of the numerical integration method used.

\section{Computational Examples and Results}\label{sec: examples}

This section is dedicated to describing the code used for implementing \eqref{finitedifference} and certain relevant numerical examples computed using this code. We first note that empirically when $\Delta t$ and $\Delta x$ are chosen to satisfy \ref{errorConditions} then $\Delta^n_j U$ stays above $\frac{1}{2} \delta_1$. Additionally from \cite{Kitagawa12}, we know that \eqref{POT} converges exponentially to the actual solution of the optimal transport problem. In the following examples this fast convergence can be observed as the results are graphed over time using a uniformly spaced color gradient. Exponential convergence is observed due to the relatively small change in approximation at later time steps.  

In testing our code for functionality, we attempted to cover a variety of situations using appropriate choices of $f$ and $g$. Some of the more interesting cases tested have been shown here. For simplicity all cases were run with initial choice $u_0(x)=\frac{1}{2}x^2$. 
We chose not to graph the function $U^n_j$, as the function that is relevant for the optimal transport theory is $\nabla^n_j U$, which is meant to approximate the function $u'(x)$ for $u$ solving \eqref{OT}. We note that the theory for \eqref{POT} only guarantees convergence to the solution of \eqref{OT} when $f(x)$, $g(y)$  are continuous and bounded away from zero and infinity on $[A,B]$ and $[C,D]$, respectively. Some of our numerical examples test the limits of the theory by considering cases where $g$ is only piece-wise continuous and also where $f$ gets very close to zero. 

Before discussing the examples we will briefly discuss the algorithm. The full implementation in python is available at \url{https://github.com/manuelarturosantana/ParabolicOptimalTransport}

 \subsection{Algorithm}
 
\begin{algorithm}
\KwResult{Returns the Approximated Solution of OT }
 Calculate $\Delta x, \Delta t_0, \delta_1, \delta_2, \Psi, K, \Gamma$\\
  current row = initial row based off $v_0$\\
 \While{Max Error $>$ Tolerance}{
  Calculate $\Delta t_n$\\
  current row = calculate next row\\
  Calculate Max Error
 }
 \caption{Simple Finite Difference Algorithm}
\end{algorithm}

Calculating each row follows the finite difference scheme \eqref{finitedifference}. The boundaries and the interior points are calculated separately. Checking the error at every grid point for every time-step is computationally expensive. To combat this, a subset of the spatial grid points, which we denote by $\mathcal{G}_x$, is selected, and at each time step only grid points in $\mathcal{G}_x$ are tested to be within tolerance. If at a certain time step all grid points in $\mathcal{G}_x$ are within tolerance, we proceed by calculating the error at \emph{all} spatial grid points from that time step onward until tolerance is reached at every grid point. 


We now describe several computational examples. In examples \ref{simple}, \ref{hiFreqEx}, \ref{otherHiFreqEx}, \ref{920Ex}, \ref{PieceWiseEx} the domains $[A,B] = [C, D] = [-1,1]$. In \ref{QuantileEx} $[A,B] = [0,1]$ and $[C,D] = [-\pi, \pi].$ The graphs with color maps show the solution plotted every $1000$ iterations. In the tables below $t_{total}$ represents the value of $t$ in \ref{POT} at the final time-step, and CPU Time represents the computer run time in seconds. All were run on the CoCalc cloud computing environment with the academic research package. 

\subsection{`Nice' Functions}
In testing our code we tried an initial variety of computationally nice functions for both $f$ and $g$. These functions are bounded well above 0 ($>0.1$), continuous, and did not change convexity more than twice. Such examples include logarithmic, exponential, linear, quadratic, constant, and concave cosine functions all modified to fit the conditions of \eqref{POT}.  The following is an example of numerical output using two functions from this set:
\begin{example}\label{simple}
\begin{equation*}
    f(x) = \frac{\log(x+2)}{3\log(3)+ 2} + \frac{2}{3\log(3) + 2}, \quad g(x) = \frac{1}{2}x^2 + \frac{1}{3}
\end{equation*}
\end{example}

\begin{figure}[H]
    \centering
    \begin{subfigure}{.5\linewidth}
  \centering
  \includegraphics[width = .95\linewidth]{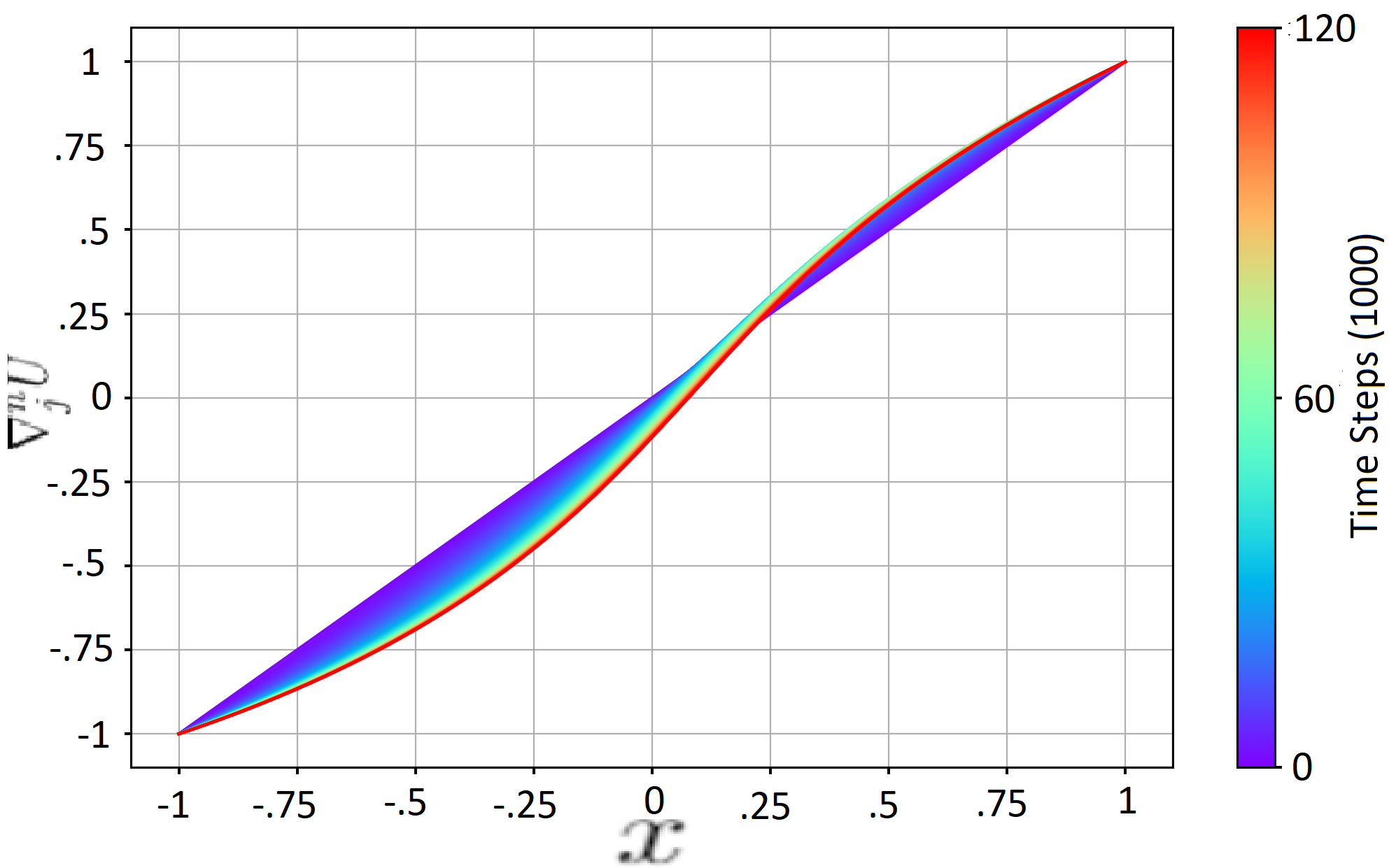}
  \caption{$\na U$ Over Time}
\end{subfigure}%
\begin{subfigure}{.5\linewidth}
  \centering
  \includegraphics[width = .95\linewidth]{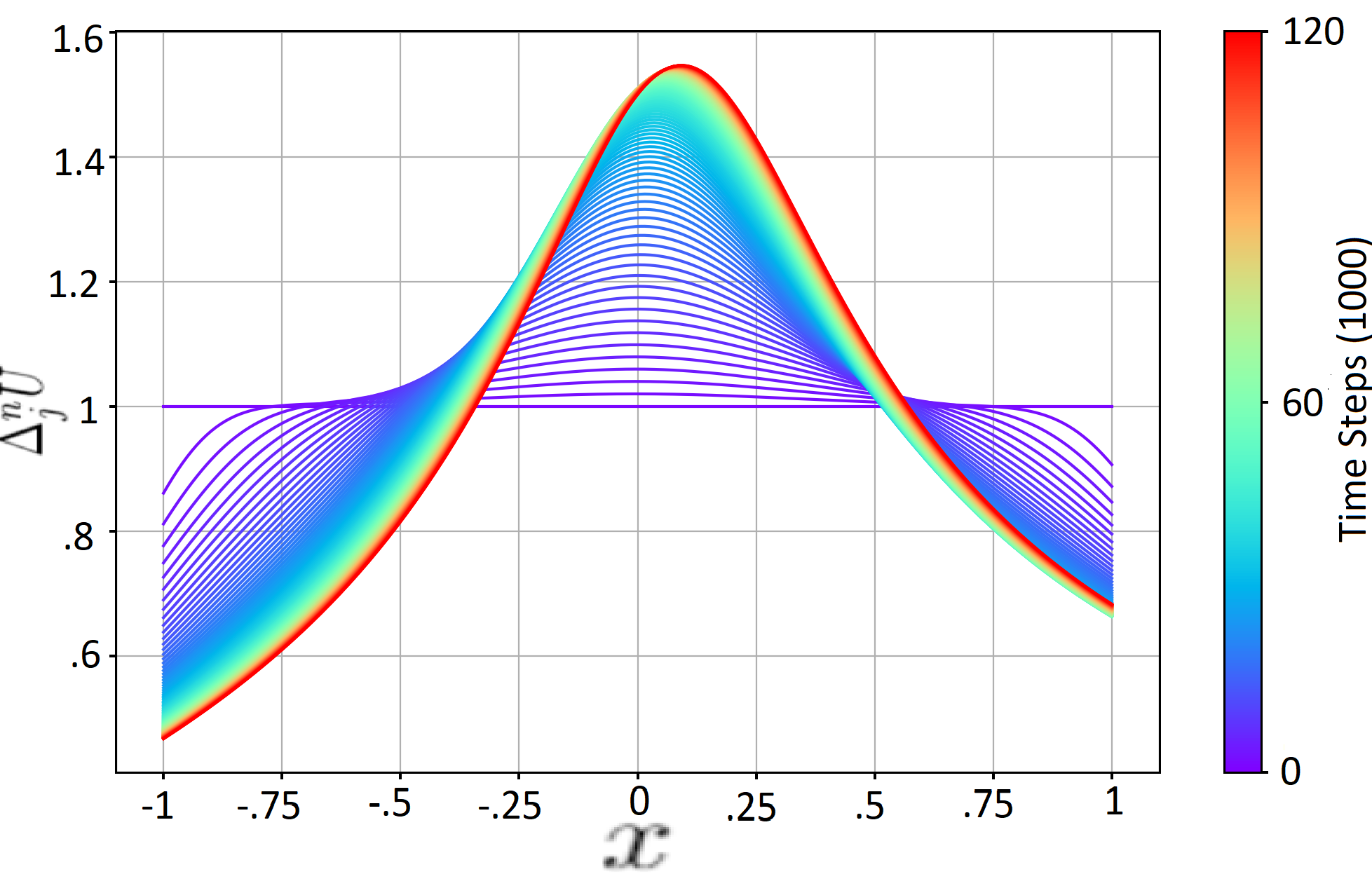}
  \caption{$\Delta^n_j U$ Over Time}
\end{subfigure}
    \caption{Graphs of \ref{simple}, 152.9s to Reach Tolerance \\
    $\epsilon = 0.01, \quad \max\limits_{j=0,\ldots,J} |T(x_j) - \nabla^n_j U| \leq  0.03$}
    \label{fig:simple}
\end{figure}

\begin{table}[H]
    \centering
    \begin{tabular}{||c|c|c|c||}
      \hline
    Tolerance & Iterations & $t_{total}$ & CPU Time (s) \\
     \hline \hline
     0.1&814&0.0053&1.06 \\
     \hline
     0.01&119880&0.7815&150 \\
     \hline
     0.001&289020&1.88&360 \\
     \hline
     0.0001&459807&2.997&573 \\
     \hline
    \end{tabular}
    \caption{The computational time and iterations to reach tolerance. The sum of time steps is also given.}
    \label{tab:simple}
\end{table}
As expected the graph in Figure \ref{fig:simple} shows exponential convergence to the optimal map.

\subsection{High Frequency Functions}
\begin{example}\label{hiFreqEx}
\begin{equation*}
    f(x) = \frac{50}{\sin(100)+ 200}(\cos(100x)+2), \quad g(x) = \frac{1}{4} (x+2)
\end{equation*}
\end{example}
This case has frequent convexity changes of the initial mass function, $f(x)$. Testing this case allows us to know that our our code is able to handle more complex smooth cases. 

\begin{figure}[H]
\begin{subfigure}{.5\linewidth}
  \centering
  \includegraphics[width = .95\linewidth]{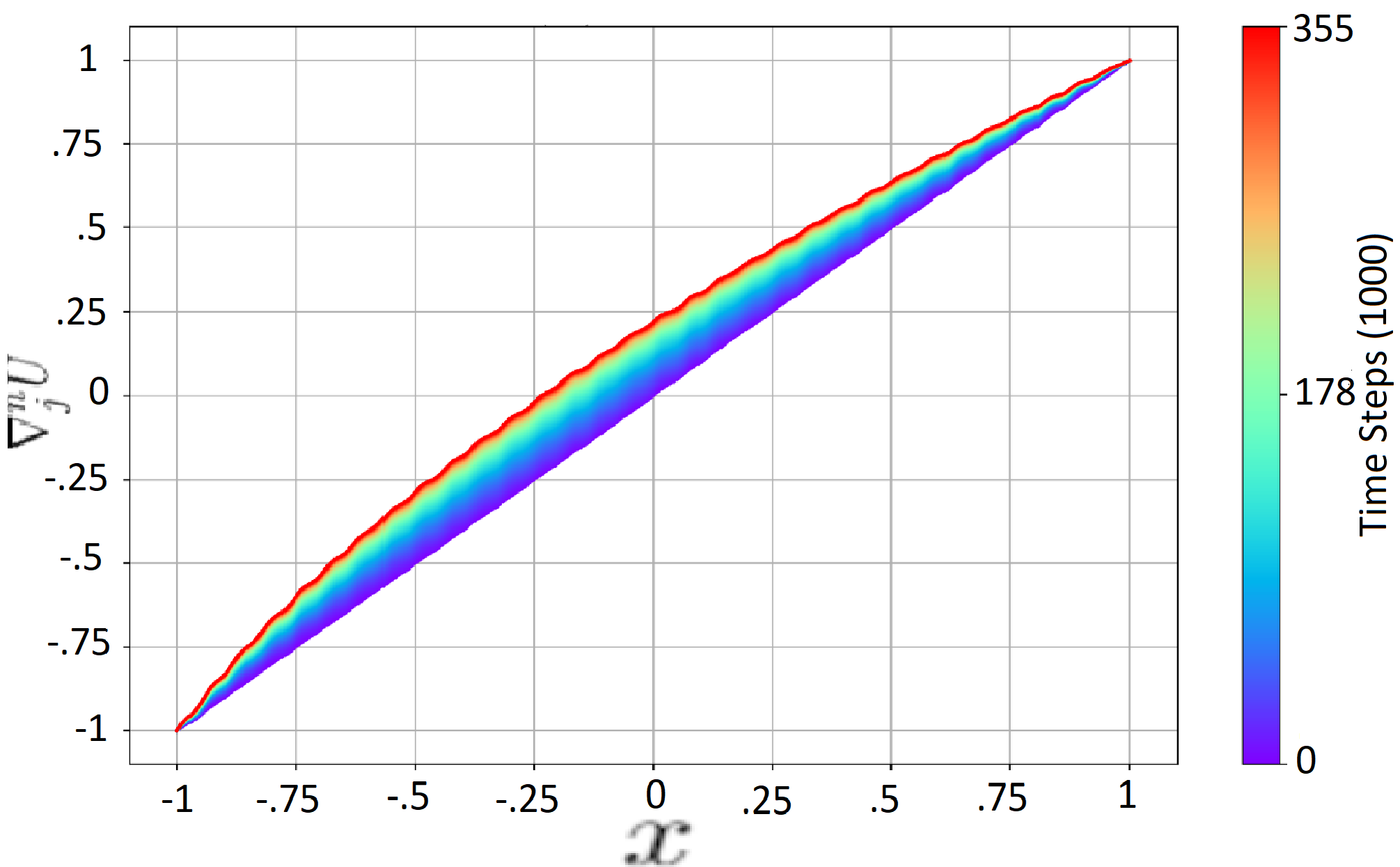}
  \caption{$\na U$ Over Time}
\end{subfigure}%
\begin{subfigure}{.5\linewidth}
  \centering
  \includegraphics[width = .95\linewidth]{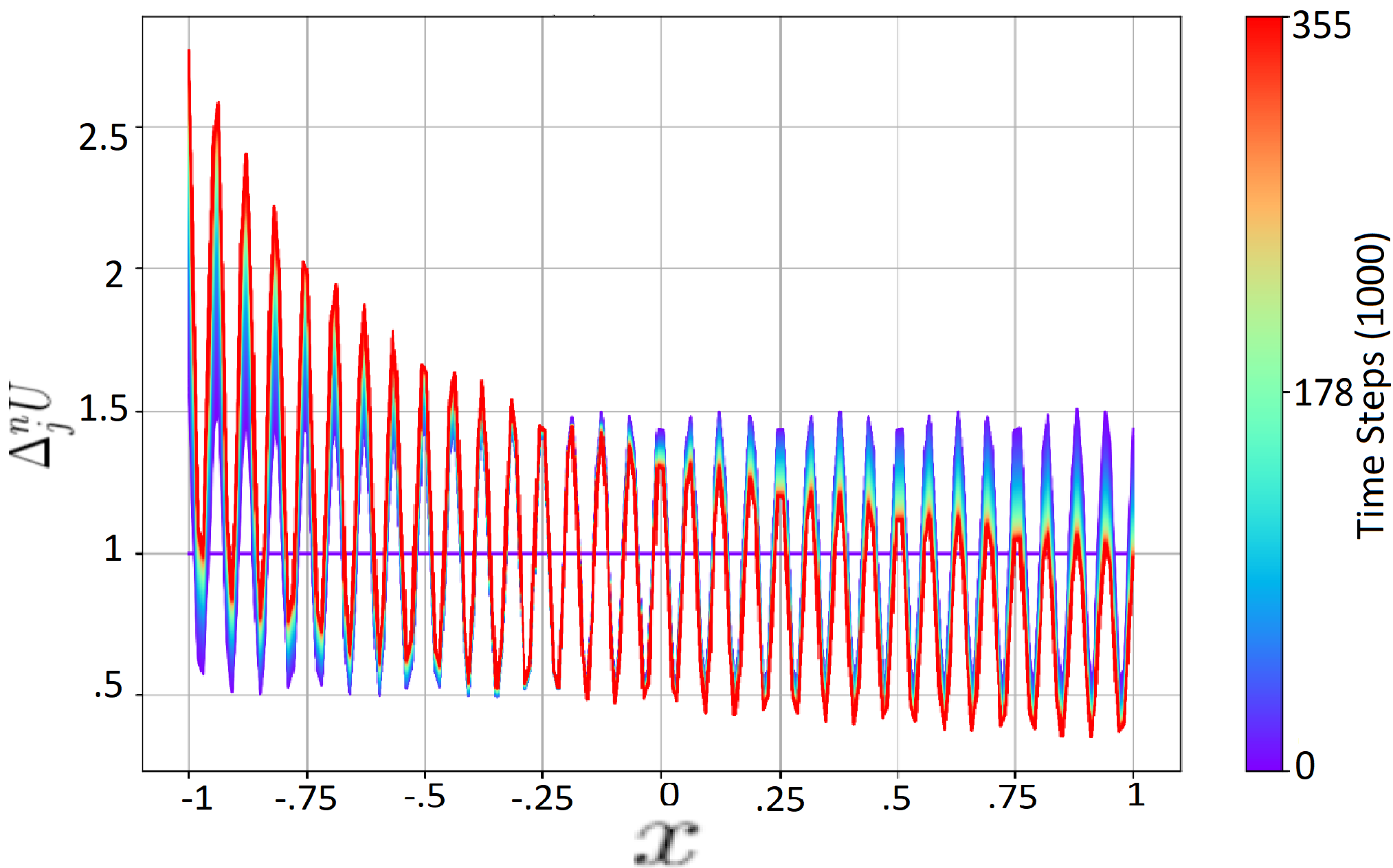}
  \caption{$\Delta^n_j U$ Over Time}
\end{subfigure}
\caption{Graphs of \ref{hiFreqEx}, 430s to Reach Tolerance \\
$\epsilon = 0.01, \quad \max\limits_{j=0,\ldots,J} |T(x_j) - \nabla^n_j U| \leq  0.04$}
\label{fig:hiFreq}
\end{figure}

\begin{table}[H]
    \centering
    \begin{tabular}{||c|c|c|c||}
      \hline
    Tolerance & Iterations & $t_{total}$ & CPU Time(s) \\
     \hline \hline
     0.1&35877&0.1006&41.31 \\
     \hline
     0.01&354534&0.9947&414 \\
     \hline
     0.001&692928&1.944&817\\
     \hline
    \end{tabular}
    \caption{Numerics for high frequency function to quadratic.}
    \label{tab:hiFreq}
\end{table}

\begin{example}\label{otherHiFreqEx}
\begin{equation*}
    f(x) = \frac{1}{4} (x+2), \quad g(x) = \frac{50}{\sin(100)+ 200}(\cos(100x)+2)
\end{equation*}
\end{example}
This example switches $f(x)$ and $g(x)$ in Example \ref{hiFreqEx}. According to the optimal transport theory, the corresponding optimal map will be the inverse of the optimal map from Example \ref{hiFreqEx}. We also expected the runtime to be longer in this case, due to the high oscillation in the term $g(\nabla^n_j U)$ from \eqref{finitedifference}. 
\begin{figure}[H]
\begin{subfigure}{.5\linewidth}
  \centering
  \includegraphics[width = .95\linewidth]{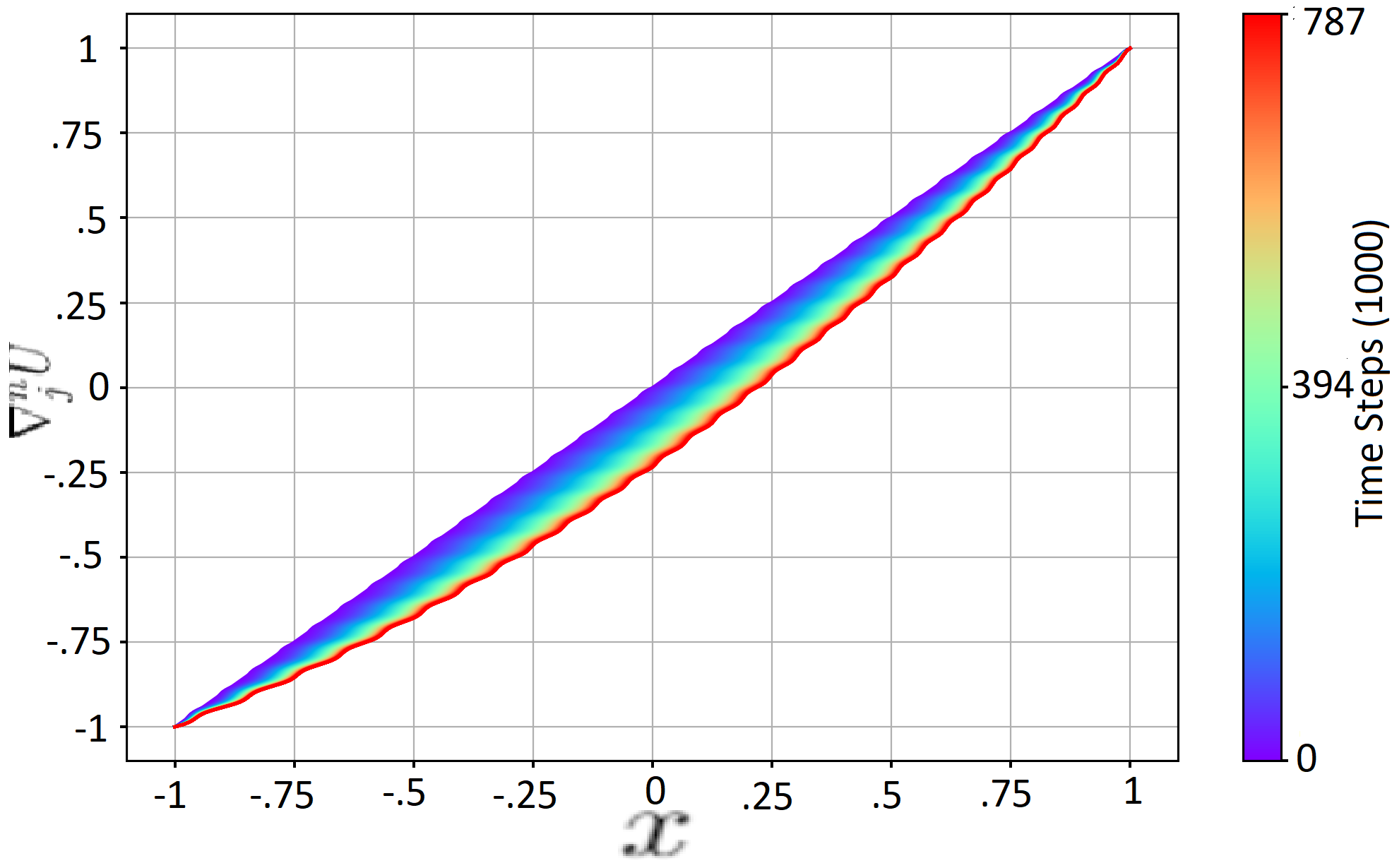}
  \caption{$\na U$ Over Time}
\end{subfigure}%
\begin{subfigure}{.5\linewidth}
  \centering
  \includegraphics[width = .95\linewidth]{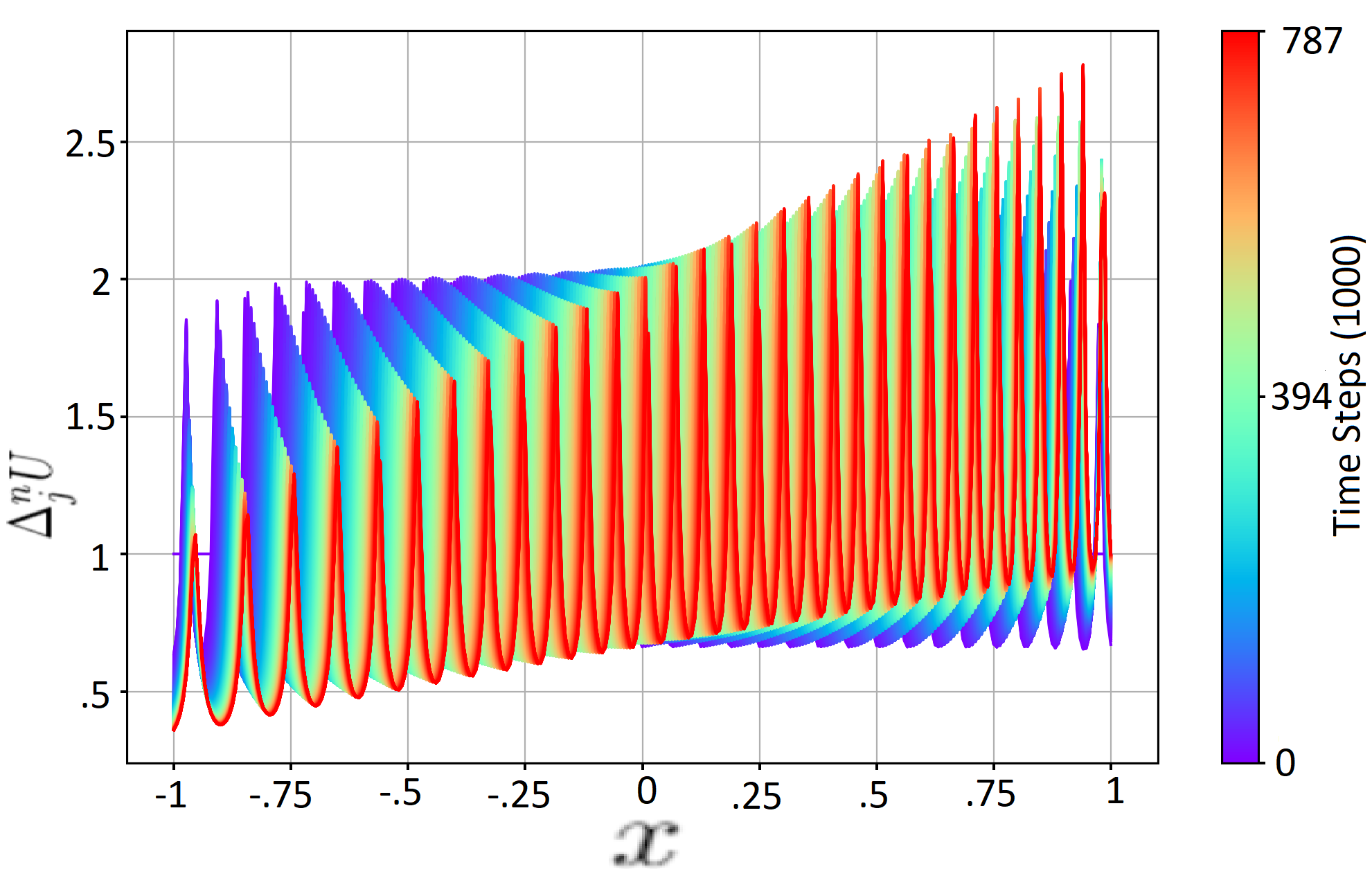}
  \caption{$\Delta^n_j U$ Over Time}
\end{subfigure}
\caption{Graphs of \ref{hiFreqEx}, 5195s to Reach Tolerance}
\label{fig:otherHiFreq}
\end{figure}

\begin{table}[H]
    \centering
    \begin{tabular}{||c|c|c|c||}
      \hline
    Tolerance & Iterations & $t_{total}$ & CPU Time(s) \\
     \hline \hline
     0.1& 83180 & 0.1001 & 596\\
     \hline
     0.01&786271&0.9454 & 5195\\
     \hline
     0.001 & 1254499 & 1.509 & 8635\\
     \hline
    \end{tabular}
    \caption{Numerics for quadratic function to high frequency.}
    \label{tab:hiFreqother}
\end{table}

As expected, this example required more computational time and iterations to reach tolerance. Furthermore, we observe that the graph of $\nabla^n_j U$ in (Figure \ref{fig:otherHiFreq}(a)) is the inverse of the graph of $\nabla^n_j U$ in (Figure \ref{fig:hiFreq}(a)), which is predicted by the optimal transport theory. Due to the large differences in computational time, it would likely be more efficient to let the initial mass distribution function $f$ be the more complicated one, and then computing the inverse of the optimal map between $f$ and $g$ if that is what one needs. However, it is worth keeping in mind the limitations of inverting a grid function, as the inverse is not necessarily defined on a well distributed set of grid points. This is illustrated in the next example. 

\subsection{Mapping Functions That Are Close To Zero}
\begin{example}\label{920Ex}
$$f(x) = \frac{9}{20}x +\frac{1}{2}, \quad g(x) = \frac{1}{2}$$
\end{example}
This is a case where the minimum of the initial mass distribution, $f(x)$, is close to 0. Although the theory implies that any smooth function bounded away from 0 will work for $f$ and $g$, cases such as this cause the error to become large and the code to fail unless we incorporate the error conditions \eqref{errorConditions} into our code. After the conditions \eqref{errorConditions} were properly incorporated into our code, we found that the finite difference scheme is able to converge within tolerance to the real solution of \eqref{OT}. 


\begin{figure}[H]
    \centering
    \begin{subfigure}{.5\linewidth}
      \includegraphics[width = .95\linewidth]{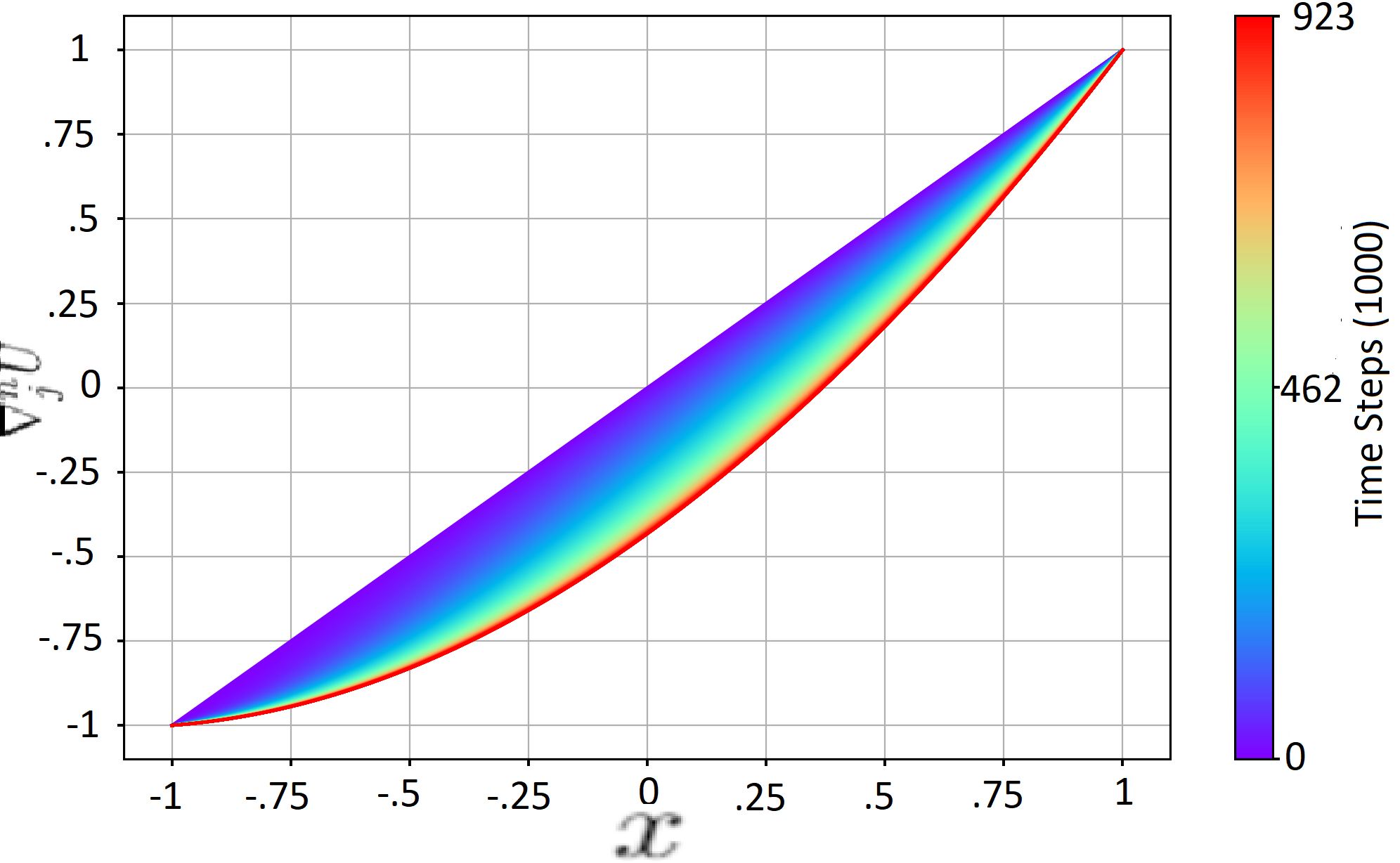}
      \caption{$\na U$ Over Time}
    \end{subfigure}%
    \begin{subfigure}{.5\linewidth}
      \includegraphics[width = .95\linewidth]{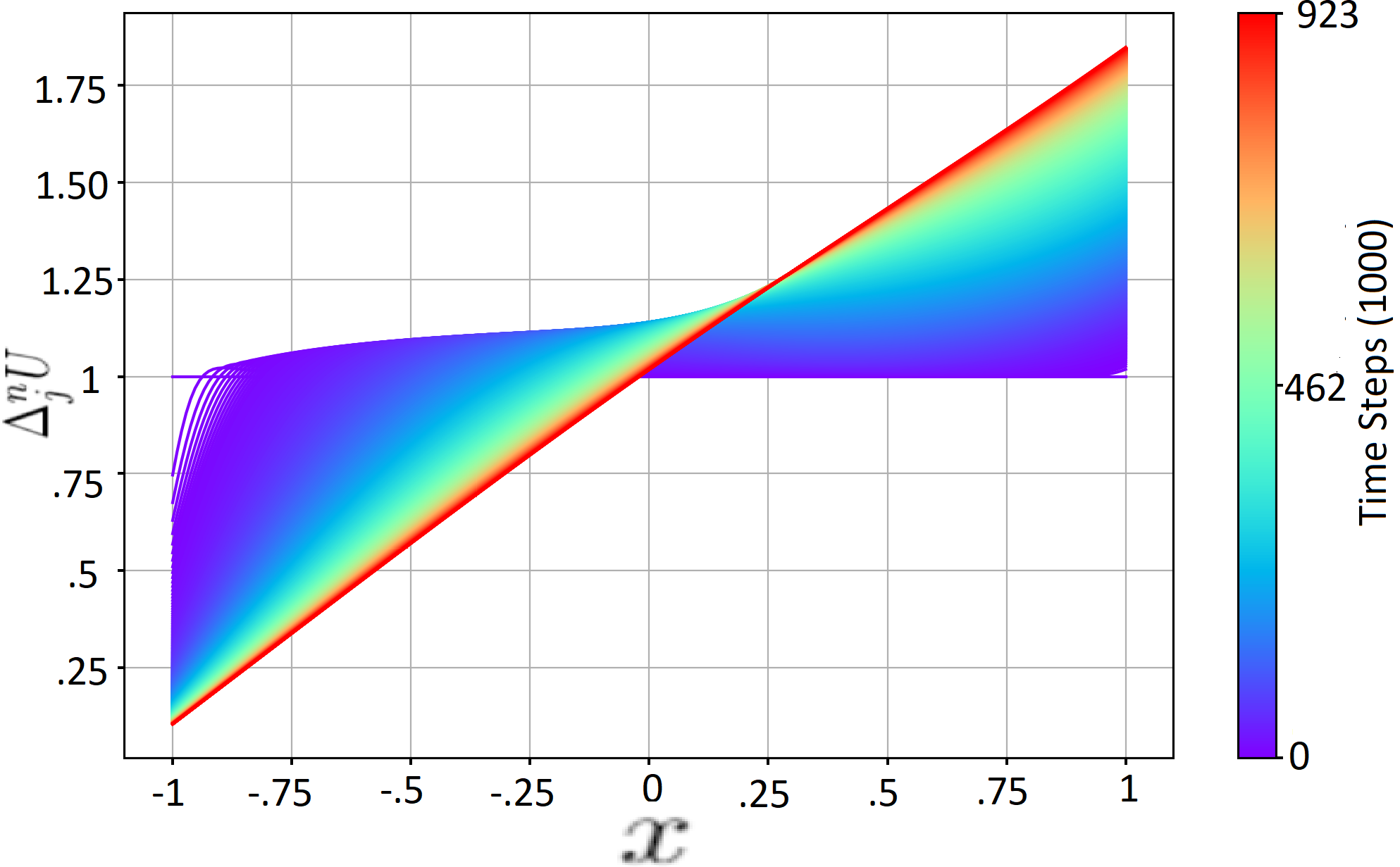}
      \caption{$\Delta^n_j U$ Over Time}
    \end{subfigure}
    \caption{Graphs of \ref{920Ex}, 985.7s to Reach Tolerance \\
    $\epsilon = 0.01, \quad \max\limits_{j} |T(x_j) - \nabla^n_j U| \leq  0.02$}
    \label{fig:920th}
\end{figure}

\begin{table}[H]
    \centering
    \begin{tabular}{||c|c|c|c||}
      \hline
    Tolerance & Iterations & $t_{total}$ & CPU Time (s) \\
     \hline \hline
     0.1&263683&0.3087&280 \\
     \hline
     0.01&922803&1.0806&975 \\
     \hline
     0.001&1568238&1.8364&1653\\
     \hline
     0.0001&2212321&2.59074&2326 \\
     \hline
    \end{tabular}
    \caption{Numerics for Near Zero Function}
    \label{tab:near0}
\end{table}

Now consider the more general case of \eqref{920Ex}; $f : [-1, 1] \to \mathbb{R}$ such that $f(x) = \beta x + \frac{1}{2}$. If $\beta$ is very close to $\frac{1}{2}$ then $\min f$ is very close to zero. It follows from Theorem \ref{theo: error bound} that $\Delta t$ gets very close to zero. In practice this has meant millions of iterations to reach convergence and a very slow program.

\subsection{Mapping Piecewise Functions}
In this section we discuss results involving mass distributions that are not guaranteed by \eqref{POT} to converge to the solution of \eqref{OT}. Yet, experimentally, with our finite difference scheme, we were able to show for some of these examples that the numerical solution can approach a desired tolerance and hence is close to the solution of \eqref{OT}. The following example uses a piecewise constant function to show this.  

\begin{example}\label{PieceWiseEx}
$$f(x) = \begin{cases}
0.3, \quad \text{if } x \le -0.5 \\
0.6, \quad \text{if } -0.5 < x \le 0 \\
0.2, \quad \text{if } 0 < x \le 0.5 \\
0.9, \quad \text{if } x > 0.5
\end{cases}, \quad g(x) =  \frac{1}{2}$$
\end{example}

\begin{figure}[H]
    \centering
    \begin{subfigure}{.5\linewidth}
      \includegraphics[width = .95\linewidth]{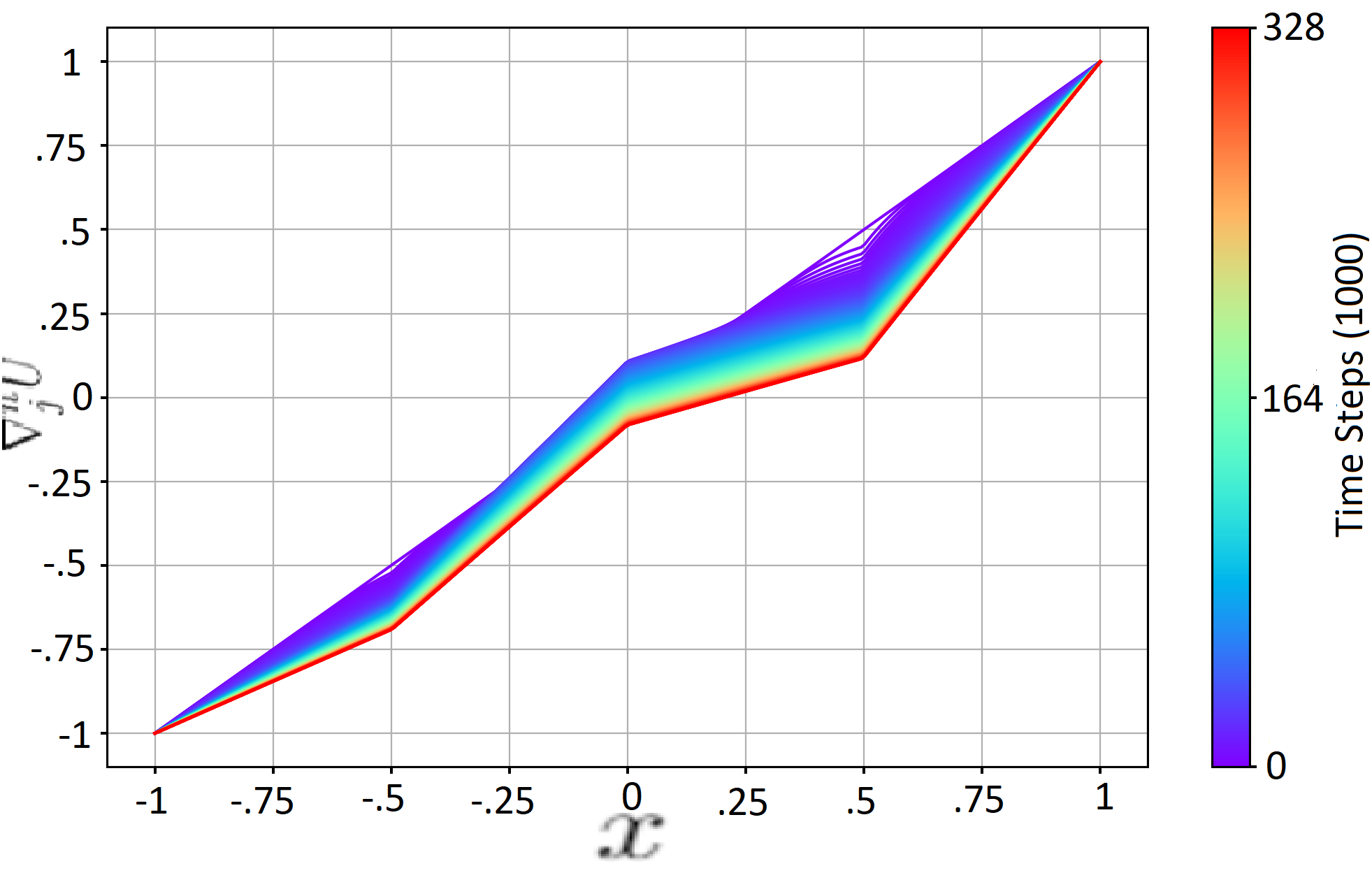}
      \caption{$\na U$ Over Time}
    \end{subfigure}%
    \begin{subfigure}{.5\linewidth}
      \includegraphics[width = .95\linewidth]{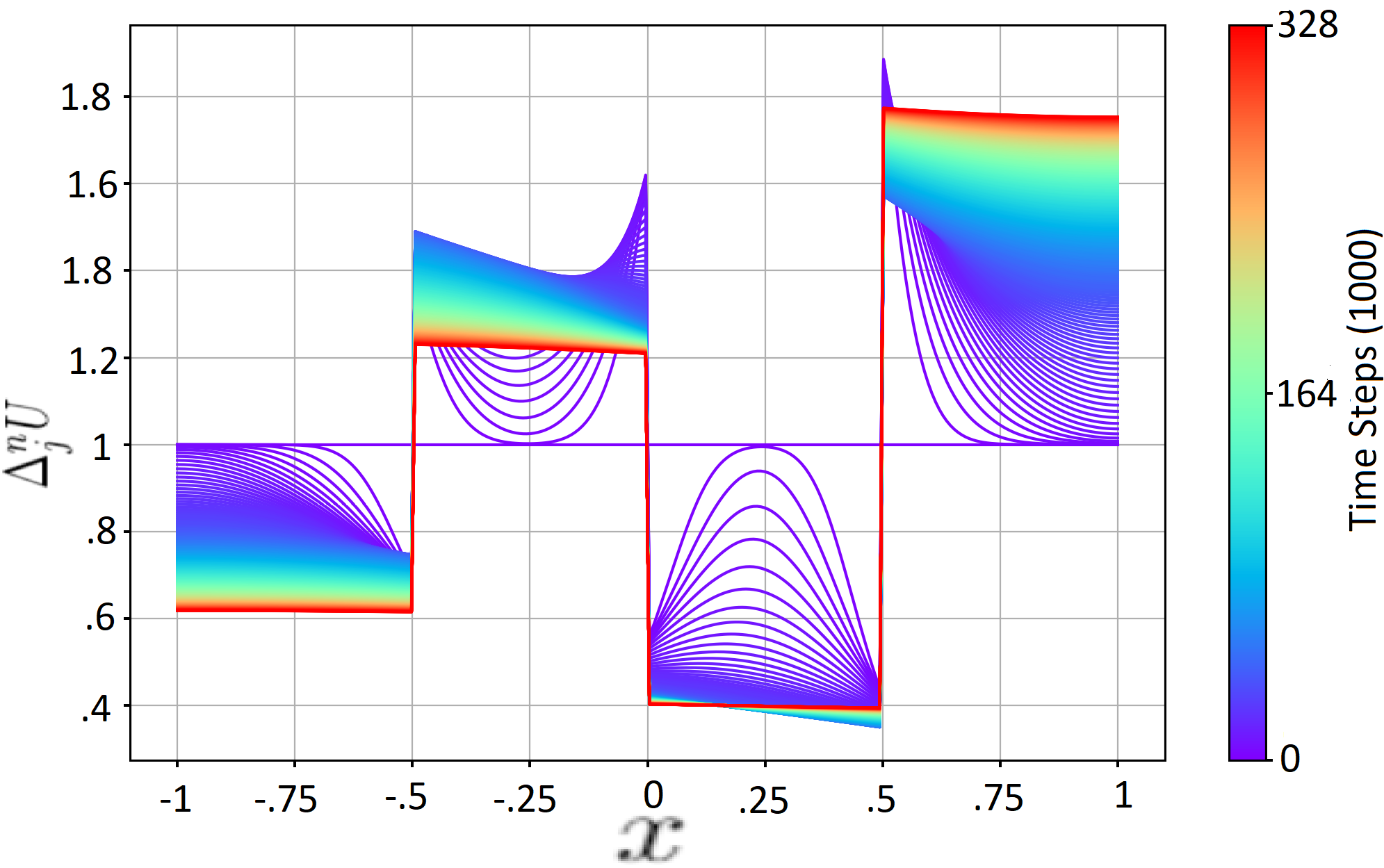}
      \caption{$\Delta^n_j U$ Over Time}
    \end{subfigure}
    \caption{Graphs of \ref{PieceWiseEx}, 1211.4s to reach tolerance \\
    $\epsilon = 0.01, \quad \max\limits_{j} |T(x_j) - \nabla^n_j U| \leq  0.021$}
    \label{fig:pieces}
\end{figure}

\begin{table}[H]
    \centering
    \begin{tabular}{||c|c|c|c||}
      \hline
    Tolerance & Iterations & $t_{total}$ & CPU Time (s) \\
     \hline \hline
     0.1&31325&0.1088&38.47 \\
     \hline
     0.01 & 327070 & 1.1356 & 400 \\
     \hline
     0.001 & 672739 & 2.3356 & 821\\
     \hline
    \end{tabular}
    \caption{Numerics for Piecewise Function}
    \label{tab:piecesTab}
\end{table}
We were also motivated to test whether a piecewise function would be able to converge within tolerance when mapped to another piecewise function, with discontinuities at different points. Furthermore, we wanted to see the effects of functions that were not piecewise constant. This led us to test the following example:
\begin{example}\label{morePieces}\begin{equation*}
    f(x) = \begin{cases}
\frac{1}{2}\log(x+2) + \frac{13}{12} - 2\log(2), &\text{if } x \le 0 \\
\frac{1}{4}x^2 + \frac{1}{3},  &\text{if } x > 0
\end{cases}, \quad 
g(x) = \begin{cases}
\frac{3}{10}x + \frac{7}{10}, \quad &\text{if } x \le \frac{-1}{3} \\
\frac{1}{2}, \quad &\text{if } \frac{-1}{3} < x \le \frac{1}{3}\\
\frac{-3}{10}x + \frac{7}{10}, \quad &\text{if } \frac{1}{3} < x
\end{cases}
\end{equation*}
\end{example}

\begin{figure}[H]
    \centering
    \begin{subfigure}{.5\linewidth}
      \includegraphics[width = .95\linewidth]{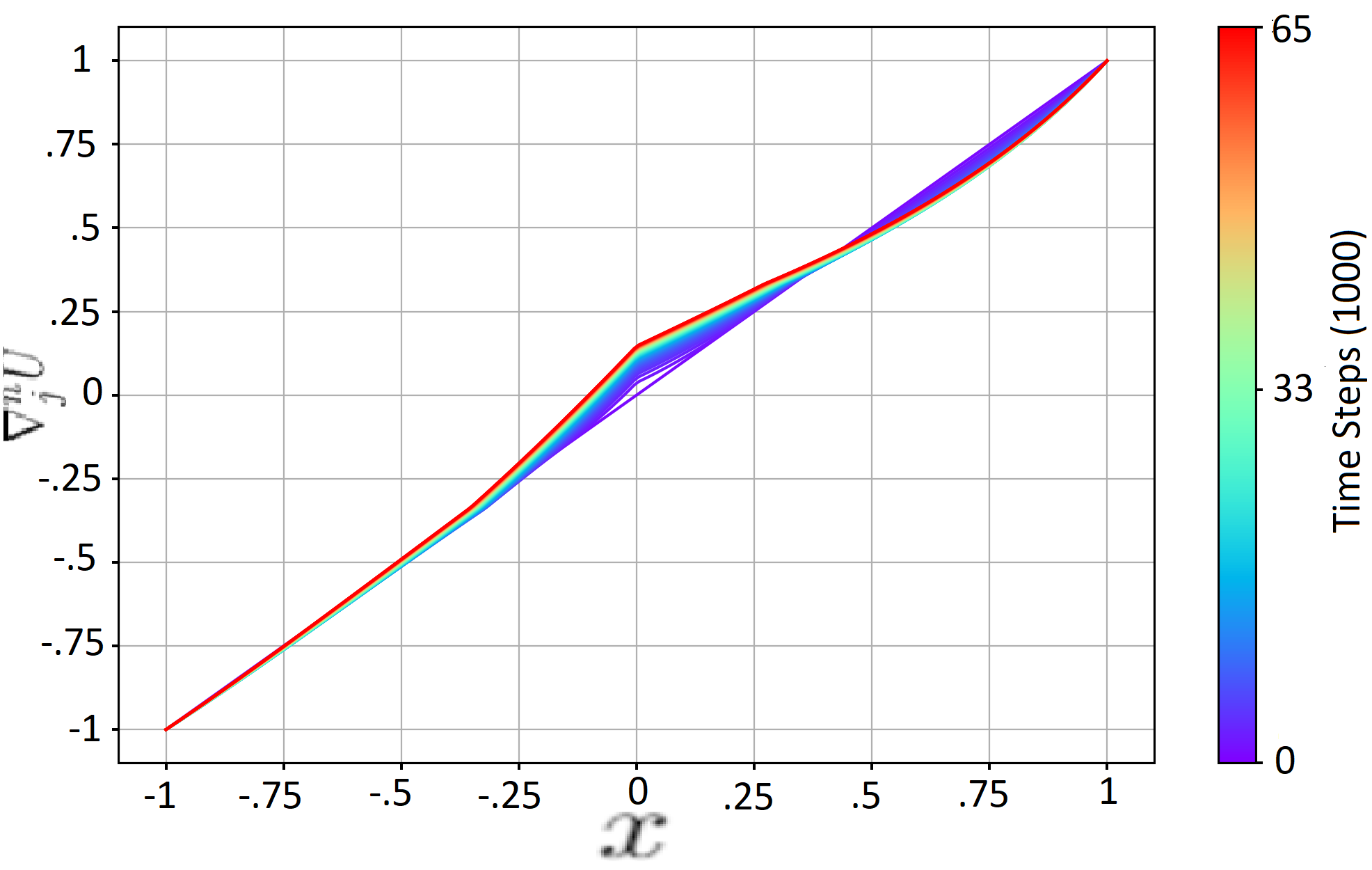}
      \caption{$\na U$ Over Time}
    \end{subfigure}%
    \begin{subfigure}{.5\linewidth}
      \includegraphics[width = .95\linewidth]{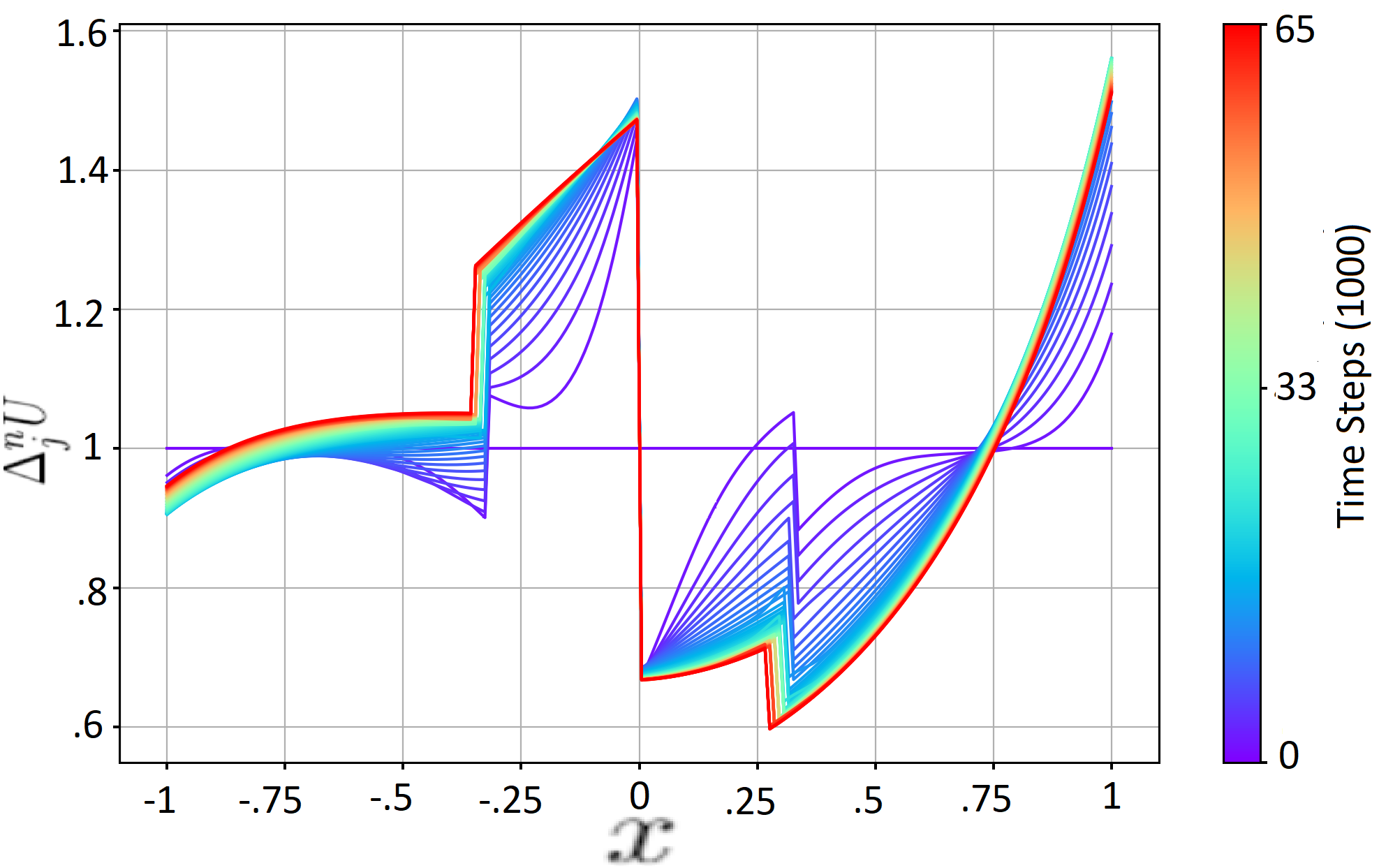}
      \caption{$\Delta^n_j U$ Over Time}
    \end{subfigure}
    \caption{Graphs of \ref{morePieces}, 356s to reach tolerance \\
    $\epsilon = 0.01, \quad \max\limits_{j} |T(x_j) - \nabla^n_j U| \leq  0.025$}
    \label{fig:pieces2}
\end{figure}

\begin{table}[H]
    \centering
    \begin{tabular}{||c|c|c|c|c||}
    \hline 
    \multicolumn{3}{||c}{ } & \multicolumn{2}{c||}{CPU Time (s)}\\
      \hline
    Tolerance & Iterations & $t_{total}$ & Numerical & Analytical \\
     \hline \hline
     0.1&1&7.616e-06&0.1283 & 0.0709\\
     \hline
     0.01&64445&0.4909&355& 71.4\\
     \hline
     0.001&184774&1.408&1166 & 202\\
     \hline
    \end{tabular}
    \caption{Numerics for Piecewise Function}
    \label{table:OtherPiecesTab}
\end{table}

From this we were able to see that our scheme seems to also converge within tolerance even when given a piecewise function for both $f$ and $g$. Although only the points of discontinuity in $f$ are seen in $\na U$, we see all points of discontinuity in both $f$ and $g$ appear in the graph of $\Delta^n_j U$. This aligns with the expectation that $\Delta^n_j U$ approximates $u''$. We observed that piecewise functions tend to take more computational time to reach tolerance compared against smooth cases with similar upper and lower bounds on $f$ and $g$. Even so, experimentally we found that closeness to zero had more of an effect on computational time. 

Standard numerical integrators can have difficulties integrating discontinuous functions accurately and efficiently. Therefore, in cases involving piecewise functions it may be necessary to alter the error tolerance methods. One solution may be to implement function for the exact integral, calculated analytically if possible. Another would be to implement a specialized numerical integrator capable of handling piecewise functions. Both methods can also significantly improve computational time. See Table \ref{table:OtherPiecesTab} for CPU time differences of the standard numerical integrator and exact analytical integrator for \ref{morePieces}. 

\subsection{Quantile Example}
Note that \eqref{quantileInverse} provides a way to calculate the inverse of $G(x)$ if $F(x) = x$; that is, if we let $f \sim \text{Unif}[A,B]$. The numerical scheme \eqref{finitedifference} thus provides a way to compute the quantile function of any probability distribution that is supported on a bounded interval and stays away from zero.
\begin{example}\label{QuantileEx}
$$f(x) = 1, \quad g(x) = \frac{e^{\kappa\cos(x)}}{2 \pi \mathcal{I}_0(\kappa)}$$
\end{example}
The function $g$ is an example of a von Mises distribution with $\mu = 0$ and $\kappa = 1$. Recall that $\mathcal{I}_0$ is the modified Bessel function of the first kind and of order zero. In this case $f:[0,1] \to \mathbb{R}$ and $g:[-\pi, \pi] \to \mathbb{R}$ with $u_0 = \pi(x^2 - x)$. \\

\begin{figure}[H]
    \centering
    \begin{subfigure}{.5\linewidth}
      \includegraphics[width =.95\linewidth]{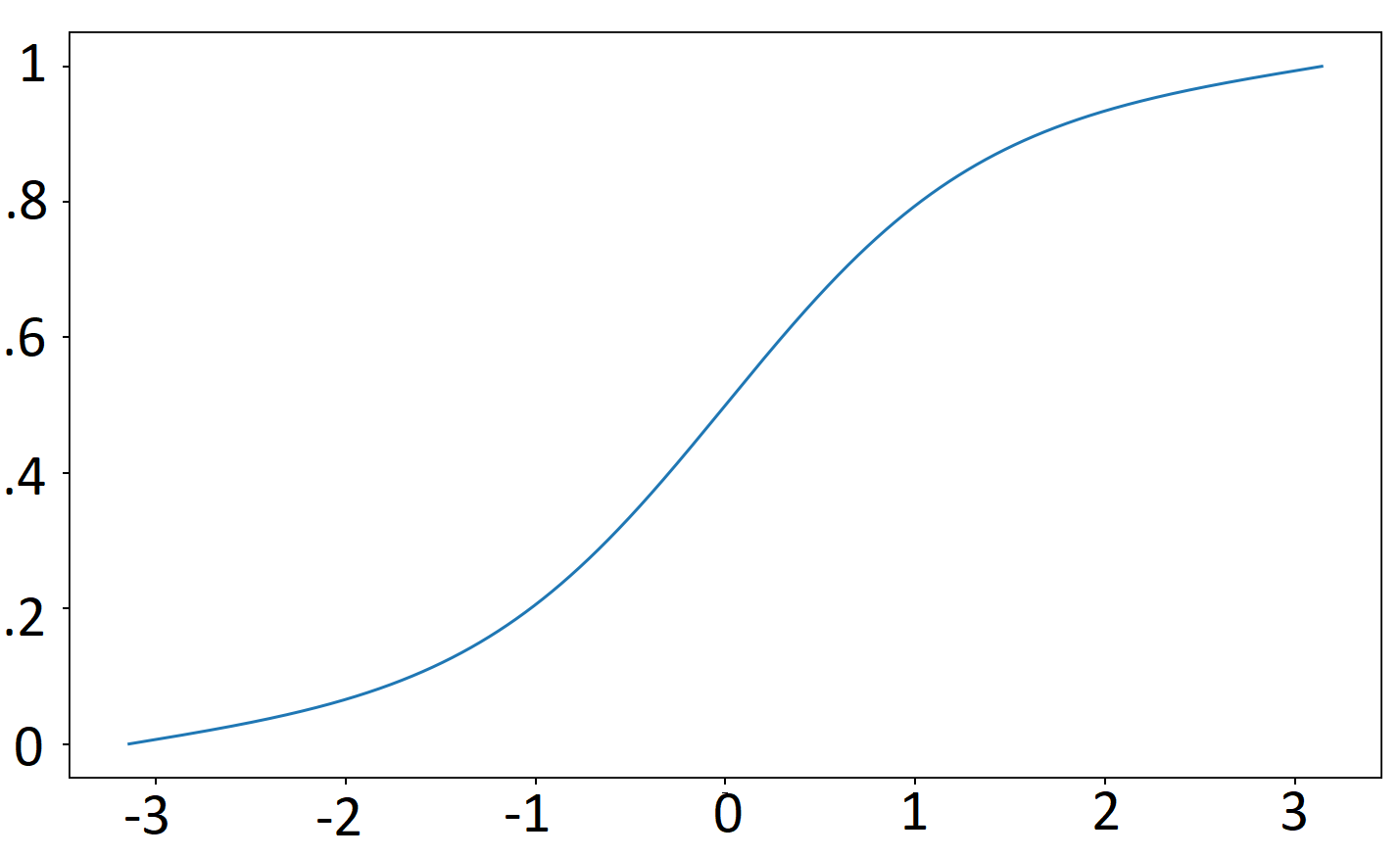}
      \caption{CDF of von Mises Distribution}
    \end{subfigure}%
    \begin{subfigure}{.5\linewidth}
      \includegraphics[width = .95\linewidth]{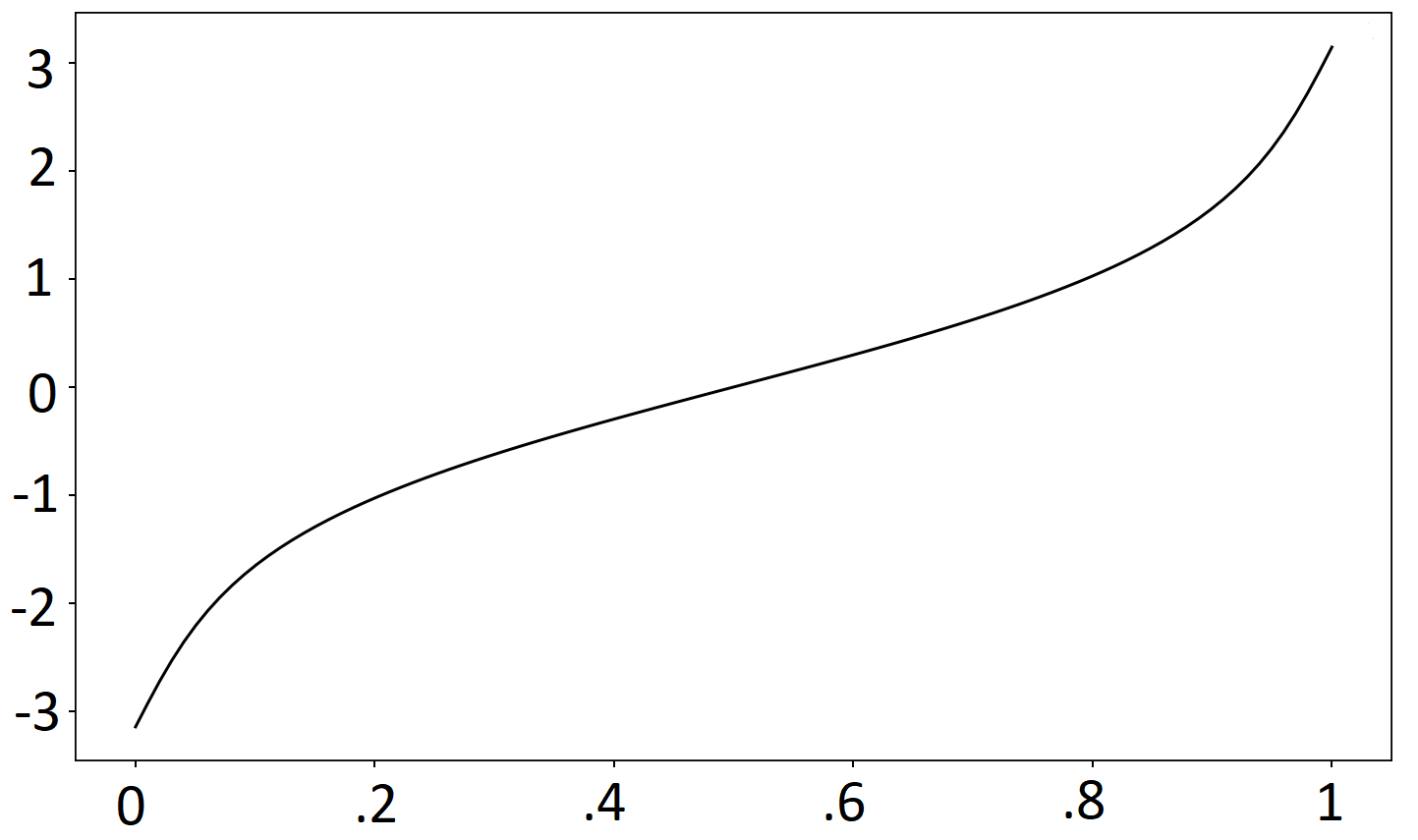}
      \caption{Quantile of von Mises Distribution}
    \end{subfigure}
    \caption{Graphs  \ref{QuantileEx}, 801.8s to Reach Tolerance \\
    $\epsilon = 0.001, \quad \max\limits_{j=0,\ldots,J} |T(x_j) - \nabla^n_j U| \leq  0.022$}
    \label{fig:Quantile}
\end{figure}

The usual method for numerically approximating a quantile function involves numerically integrating the probability density function, $g$ to get an approximation of the cumulatively distribution function $G$, then inverting the $x$ and $y$ coordinates of $G$. 
Inverting a grid function can result in a non-uniform grid for the numerical inverse, and if $G$ has large derivative, then the grid points of the domain of $G^{-1}$ will be concentrated along the points where the derivative of $G$ is large. Using \ref{finitedifference} to compute an approximation of $G^{-1}$ does not involve inverting a grid function, and therefore provides better resolution, though it is computationally much slower.

\section*{Conclusion and Outlook}
 
 We have shown error bounds on the finite difference scheme for the 1-D parabolic optimal transport problem and provided relevant numerical examples.
If the error conditions \eqref{errorConditions} are met, we have seen empirically that $\Delta^n_j U$ stays positive when calculating the optimal map, allowing for the application of Theorem \ref{theo: error bound}. We hope to further investigate this condition and prove that $\Delta^n_j U$ always stays positive in future work. 
 
 The error given by \eqref{espError} provides a way to quantify the error at any given time step $n$, but  it is not guaranteed to stay small as $n\to \infty.$  In practice, this is not detrimental to the efficacy of \eqref{finitedifference} as the scheme is only run for a finite number of time steps until it is within tolerance. We also hope to be able to bound the error in \eqref{finitedifference} for all $n$ in future work.
 
We have a first order term in the error bound at the boundary, which adversely impacts our schemes accuracy in approximating \eqref{POT}. This does not seem to have much impact on the accuracy our scheme in  approximating the optimal map as it is empirically  it is still able to get within tolerance. In the future we hope to investigate how to replace this first order error term with a second order one in order to create a scheme that can more accurately approximate \eqref{POT}. 

 Additionally our research was only carried out for one spatial dimension. Further research is necessary to devise robust numerical methods for the optimal transport problem in two dimensions and higher. Additional work needs to be done to understand why piecewise functions used in the \eqref{POT} are able to converge to \eqref{OT} and proven mathematically.

\section*{Acknowledgements}
This research was supported by the following grants; NSA Award No. H98230-20-1-0006 and NSF Award No. 1852066. We would like to thank Professor Robert Bell and Michigan State University for organizing SURIEM, and we would like to thank our research mentors Farhan Abedin and Jun Kitagawa for their guidance.

\bibliography{MAarxiv}

\begin{thebibliography}{1}

\bibitem{AbedinKitagawa}
Farhan Abedin and Jun Kitagawa.
\newblock Exponential convergence of parabolic optimal transport on bounded
  domains.
\newblock \emph{Anal. PDE 13 (2020), no. 7, 2183--2204.}

\bibitem{Benamou}
Jean-David Benamou, Brittany~D. Froese, and Adam~M. Oberman.
\newblock Two numerical methods for the elliptic {M}onge-{A}mp\`ere equation.
\newblock {\em M2AN Math. Model. Numer. Anal.}, 44(4):737--758, 2010.

\bibitem{Benamou2}
Jean-David Benamou, Brittany~D. Froese, and Adam~M. Oberman.
\newblock Numerical solution of the optimal transportation problem using the
  {M}onge-{A}mp\`ere equation.
\newblock {\em J. Comput. Phys.}, 260:107--126, 2014.

\bibitem{Brenier}
Yann Brenier.
\newblock Polar factorization and monotone rearrangement of vector-valued
  functions.
\newblock {\em Comm. Pure Appl. Math.}, 44(4):375--417, 1991.

\bibitem{Kitagawa12}
Jun Kitagawa.
\newblock A parabolic flow toward solutions of the optimal transportation
  problem on domains with boundary.
\newblock {\em J. Reine Angew. Math.}, 672:127--160, 2012.

\bibitem{Monge}
Gaspard Monge.
\newblock M{\'e}moire sur la the orie des d\'eblais et de remblais.
\newblock {\em Histoire de l’Acad\'emie Royale des Sciences de Paris, avec
  les M\'emoires de Math\'ematique et de Physique pour la m\^eme ann\'ee,},
  page 666–704, 1781.

\bibitem{1901.05108}
Michael Neilan, Abner~J. Salgado, and Wujun Zhang.
\newblock The {M}onge-{A}mp\`ere equation.
\newblock \emph{Handbook of Numerical Analysis, vol 21:105-219, 2020.}

\bibitem{Tadmor2012}
Eitan Tadmor.
\newblock A review of numerical methods for nonlinear partial differential
  equations.
\newblock {\em Bull. Amer. Math. Soc. (N.S.)}, 49(4):507--554, 2012.

\end{thebibliography}
\bibliographystyle{plain}
\appendix\label{bounds}
\begin{center}
  \section{Derivative Estimates}
\end{center}

By differentiating \eqref{POT} in $t$, we find that the function $w(t,x) := v_t(t,x)$ solves the linearized equation
\begin{equation}\label{linearization}\tag{L-E}
\begin{cases}
w_t - \left(\frac{1}{v_{xx}}\right) w_{xx} - \left(\frac{g'(v_x)}{g(v_x)} \right) w_x = 0 & \text{in }(0,\infty) \times (A,B),\\
w_x(t,A) = 0, \quad w_x(t,B) = 0 & \text{for all } t \geq 0.
\end{cases}
\end{equation}

\subsection{Bounds on \texorpdfstring{$v_{xx}$}{}}

Let $w = v_t$ as above. Then $w$ satisfies \eqref{linearization}. Since $v_{xx} \geq 0$, the parabolic maximum principle and Hopf's lemma implies 
$$\max_{x\in[A,B], \ t \geq 0} w(t,x) = \max_{x\in[A,B]} w(0,x),  \qquad \min_{x\in[A,B], \ t \geq 0} w(t,x) = \min_{x\in[A,B]} w(0,x).$$
In terms of $v_t$, this means
$$\max_{x\in[A,B], \ t \geq 0} v_t(t,x) = \max_{x\in[A,B]} v_t(0,x),  \qquad \min_{x\in[A,B], \ t \geq 0} v_t(t,x) = \min_{x\in[A,B]} v_t(0,x).$$
Evaluating \eqref{POT} at $t = 0$, we get
$$v_t(0,x) = \log(v_{xx}(0,x)) - \log\left(\frac{f(x)}{g(v_x(0,x))} \right) = \log(u_0''(x)) - \log\left(\frac{f(x)}{g(v_x(0,x))} \right) = \log\left(\frac{u_0''(x) g(v_x(0,x))}{f(x)} \right) $$
Consequently,
$$\min_{x \in [A,B]}\frac{u_0''(x) g(v_x(0,x))}{f(x)} \leq e^{v_t(0,x)} \leq \max_{x \in [A,B]}\frac{u_0''(x) g(v_x(0,x))}{f(x)} .$$
In particular,
\begin{align*}
\frac{\min_{x \in [A,B]} u_0''(x) \min_{y \in [C,D]} g(y)}{\max_{x \in [A,B]} f(x)} & \leq \min_{x\in[-1,1], \ t \geq 0} e^{v_t(t,x)} \\
& \leq \max_{x\in[A,B], \ t \geq 0} e^{v_t(t,x)}  \leq  \frac{\max_{x \in [A,B]} u_0''(x) \max_{y \in [C,D]} g(y)}{\min_{x \in [A,B]} f(x)}.
\end{align*}
Since \eqref{POT} implies
$$v_{xx}(t,x) = \frac{e^{v_t(t,x)} f(x)}{g(v_x)} \quad \text{for all } (t,x) \in (0,\infty) \times (-1,1),$$
we conclude that
\begin{footnotesize}
\begin{equation}\label{boundsonvxx}
\min_{x \in [A,B]} u_0''(x)\left( \frac{\min_{x \in [A,B]} f(x)} {\max_{x \in [A,B]} f(x)}\right) \left( \frac{\min_{y \in [C,D]} g(y)}{\max_{y \in [C,D]} g(y)}\right) \leq v_{xx}(t,x) \leq  \max_{x \in [A,B]} u_0''(x)\left( \frac{\max_{x \in [A,B]} f(x)} {\min_{x \in [A,B]} f(x)}\right) \left( \frac{\max_{y \in [C,D]} g(y)}{\min_{y \in [C,D]} g(y)}\right).
\end{equation}
\end{footnotesize}

\subsection{Bounds on \texorpdfstring{$v_{xxx}$}{}}

We let $w := v_t$ and $\phi := v_{xx}$. Differentiating $\eqref{POT}$ w.r.t $x$ gives us the relation
\begin{equation}\label{phiandw}
w_x = \frac{\phi_x}{\phi} - F(x) + G(v_x) \phi,
\end{equation}
where $F(x) := \frac{f'(x)}{f(x)}$ and $G(y) := \frac{g'(y)}{g(y)}$. Since $\phi$ is uniformly bounded, it follows that an estimate for $w_x$ yields an estimate for $\phi_x = v_{xxx}$ under appropriate assumptions on $f$ and $g$.

Recall the linearized equation \eqref{linearization}
$$
\begin{cases}
L(w) := w_t - \phi^{-1} w_{xx} = G(v_x) w_x & \text{in }(0,\infty) \times (A,B),\\
w_x(t,A) = 0, \quad w_x(t,B) = 0 & \text{for all } t \geq 0.
\end{cases}
$$
Differentiating the equation $L(w) = 0$ w.r.t $x$ gives
\begin{align*}
L(w_x) & = -\frac{\phi_x w_{xx}}{\phi^2} + G'(v_x) \phi w_x + G(v_x) w_{xx} \\
& =  -\frac{w_{xx}(w_x + F - G(v_x) \phi)}{\phi} + G'(v_x) \phi w_x + G(v_x) w_{xx} \\
& =  \left(2G(v_x) -\frac{(w_x + F)}{\phi}\right)w_{xx} + G'(v_x) \phi w_x.
\end{align*}
Consider the auxiliary function
$$\eta = \psi_1(w_x) + \psi_2(w),$$
where $\psi_1, \psi_2$ are functions to be determined. We then have
\begin{itemize}
\item $\eta_t = \psi_1'(w_x) w_{xt} + \psi_2'(w)w_t$
\item $\eta_x = \psi_1'(w_x) w_{xx} + \psi_2'(w)w_x,$
\item $\eta_{xx} = \psi_1''(w_x) w_{xx}^2 + \psi_1'(w_x) w_{xxx} + \psi_2''(w)w_x^2 + \psi_2'(w)w_{xx}.$
\end{itemize}
Consequently,
\begin{align*}
L(\eta) & = \eta_t - \phi^{-1}\eta_{xx} \\
& = \psi_1'(w_x) w_{xt} + \psi_2'(w)w_t - \frac{1}{\phi} \left( \psi_1''(w_x) w_{xx}^2 + \psi_1'(w_x) w_{xxx} + \psi_2''(w)w_x^2 + \psi_2'(w)w_{xx}\right) \\
& =  \psi_1'(w_x) L(w_x) +  \psi_2'(w)L(w)  - \frac{1}{\phi} \left( \psi_1''(w_x) w_{xx}^2 + \psi_2''(w)w_x^2 \right) \\
& =  \psi_1'(w_x)\left(2G(v_x) -\frac{(w_x + F)}{\phi} \right) w_{xx} + \left(\psi_1'(w_x)G'(v_x) \phi  + \psi_2'(w) G(v_x) \right) w_x - \frac{1}{\phi} \left( \psi_1''(w_x) w_{xx}^2 + \psi_2''(w)w_x^2 \right).
\end{align*}
Suppose now that $\eta$ attains a maximum value at a point $(t_0,x_0) \in [0,\infty) \times [A,B]$. We assume  $\psi_1$ is increasing and satisfies $\lim_{s \to 0} \psi_1(s) = -\infty$, and that $\psi_2$ is bounded on compact sets.

{\bf Case 1:} $t_0 \geq 0, \ x_0 = A \text{ or } B$.

In this case,  since $w$ is uniformly bounded and $w_x(t,A) = w_x(t,B) = 0$, it follows that $\lim_{x \to A^{+}}\eta(t,x) = \lim_{x \to B^{-}} \eta(t,x) = -\infty$.

{\bf Case 2:} $t_0 > 0, \ x_0 \in (A,B)$.

We have  $\eta_x(t_0,x_0) = 0$ and $L (\eta) \geq 0$ at $(t_0,x_0)$. This implies
$$\psi_1'(w_x) w_{xx} = -\psi_2'(w)w_x \quad \text{at } (t_0,x_0).$$ 
Substituting this into the equation for $L(\eta)$ yields
$$0 \leq \left[ \psi_2'(w)\left(\frac{(w_x + F)}{\phi} - G(v_x)\right) + \psi_1'(w_x)G'(v_x) \phi \right] w_x - \frac{1}{\phi} \left( \psi_1''(w_x) \frac{\psi_2'(w)^2}{\psi_1'(w_x)^2} + \psi_2''(w)\right)w_x^2 .$$
We now choose 
$$\psi_1(s) = \frac{1}{2}\log(s^2), \quad \psi_2(s) = \a s, \ \a \text{ constant}.$$
Then since
$$\psi_1'(s) = \frac{1}{s}, \quad \psi_1''(s) = -\frac{1}{s^2}, \quad \psi_2'(s) = \a, \quad \psi_2''(s) = 0,$$
we find that
$$0 \leq \left[ \a \left(\frac{(w_x + F)}{\phi} - G(v_x)\right) + \frac{1}{w_x}G'(v_x) \phi \right] w_x + \frac{\a^2 w_x^2}{\phi}.$$
Rearranging terms, we get
$$0 \leq \a(1 + \a)w_x^2 +  \a \left(F - G(v_x) \phi \right)w_x + G'(v_x) \phi^2.$$
Letting $\a = -\frac{1}{2}$ yields
$$w_x^2 +2(F- G(v_x)\phi )w_x - 4G'(v_x) \phi^2 \leq 0 \quad \text{at } (t_0,x_0) .$$
This implies
$$(w_x + F- G(v_x)\phi )^2 \leq (F - G(v_x)\phi )^2 + 4|G'(v_x)| \phi^2.$$
Consequently,
$$|w_x(t_0,x_0)| \leq |F- G(v_x)\phi| + \sqrt{ (F - G(v_x)\phi )^2 + 4|G'(v_x)| \phi^2} \leq C_1(u_0,f,g).$$
Since $\eta(x,t) \leq \eta(x_0,t_0)$, we have
$$\log(|w_x(t,x)|) \leq  \log(|w_x(t_0,x_0)|) +  \frac{1}{2}|w(t,x)- w(t_0,x_0)| \leq \log(C_1(u_0,f,g)) + \max_{x \in [A,B]} |v_t(0,x)|.$$
Exponentiating this gives
$$|w_x(t,x)| \leq C_1(u_0,f,g) e^{\max_{x \in [A,B]} |v_t(0,x)|} \quad \text{for all } (x,t) \in [A,B] \times [0,\infty).$$

{\bf Case 3:} $t_0 = 0,\ x_0 \in (A,B)$

For any $t \geq 0$ and $x \in [A,B]$
\begin{align*}
\eta(t, x) & \leq \eta(0,x_0) \\
& = \psi_1(w_x(0,x_0)) + \psi_2(w(0,x_0)) \\
& = \psi_1 \left(\frac{u_0'''(x_0)}{u_0''(x_0)} - F(x_0) + G(u_0'(x_0)) u_0''(x_0) \right) + \psi_2 (v_t(0,x_0)) \\
& = \log\left(\bigg|\frac{u_0'''(x_0)}{u_0''(x_0)} - F(x_0) + G(u_0'(x_0)) u_0''(x_0)\bigg| \right)  - \frac{v_t(0,x_0)}{2} \\
& \leq \log\left( C_2(u_0,f,g) \right) -  \frac{v_t(0,x_0)}{2}.
\end{align*}
Consequently, 
$$|w_x(t,x)| \leq C_2(u_0,f,g) e^{\max_{x \in [A,B]} |v_t(0,x)|} \quad \text{for all } (x,t) \in [A,B] \times [0,\infty).$$

\subsection{Bounds on \texorpdfstring{$v_{xxxx}$}{}}

Let $z = w_x$. Differentiating the relation \eqref{phiandw} w.r.t $x$ shows that
$$z_x = w_{xx} = \frac{\phi_{xx}}{\phi} - \frac{\phi_x^2}{\phi^2} - F'(x) + G'(v_x)\phi^2 + G(v_x) \phi_x.$$
Consequently, an estimate for $z_x$ combined with an estimate for $\phi_x = v_{xxx}$ implies an estimate for $\phi_{xx} = v_{xxxx}$ under appropriate assumptions on $f$ and $g$.  

\subsubsection{Boundary Estimate}

We first bound $|z_x|$ on the boundary. Define the linear operator
$$ \tilde{L} := \partial_t - \phi^{-1} \partial^2_{xx} - \b \partial_{x} \quad \text{where } \b := G(v_x)-\frac{\phi_x}{\phi}.$$
Let $\mu := (G(v_x))_x z$. Then $z$ satisfies the initial and boundary value problem
$$
\begin{cases}
\tilde{L}(z) & = \mu,\\
z(t, A)&=z(t, B)=0 \quad \text{for all } t \geq 0,\\
z(0, \cdot)&= z_0.
\end{cases}$$
Consider the barrier function 
$$\eta(x):=\gamma(e^{\alpha(x-B)}-1).$$
where $\alpha$, $\gamma>0$ to be determined. Notice that
$$\eta(B) = 0 \quad  \text{ and } \quad \eta(A) = \gamma(e^{-\alpha(B-A)}-1) \leq 0.$$
Since
\begin{itemize}
\item $\eta_x = \gamma \a e^{\a(x-B)} $
\item $\eta_{xx} = \gamma \a^2 e^{\a(x-B)}$
\end{itemize}
we have
$$\tilde{L}(\eta - z) = \tilde{L}(\eta) + \mu = -\phi^{-1} \gamma \a^2 e^{\a(x-B)} - \beta \gamma \a e^{\a(x-B)}  + \mu = - \gamma \a e^{\a(x-B)}(\phi^{-1}\a + \beta) + \mu.$$
Let $\a > 0$ be chosen so that $\phi^{-1}\a + \beta \geq 1$. Then
$$ - \gamma \a e^{\a(x-B)}(\phi^{-1}\a + \beta) \leq -\gamma \a e^{\a(A-B)}.$$
We can now choose $\gamma$ so that $\gamma \a e^{\a(A-B)} \geq \max|\mu|$ to get $\tilde{L}(\eta - z) \leq 0$. 

We now show that $\eta \leq z$ on the parabolic boundary. First, we have $\eta(t, B) = 0 = z(t, B)$ and $\eta(t, A) \leq 0 = z(t, A)$ for all $t>0$. Next, let $\psi(x) = \eta(0,x) - z_0(x)$. Then $\psi(B) = 0$ and $\psi'(x) = \gamma \a e^{\a(x-B)} - z_0'(x)$. By Taylor's theorem, for each $x\in [A, B]$ there exists $\xi_x\in (x,B)$ such that
$$\psi(x) = \psi(B) + \psi'(\xi_x)(x-B) = (\gamma \a e^{\a(\xi_x-B)} - z_0'(\xi_x))(x-B).$$
Now since $e^{\a(x-B)} \geq e^{\a(A-B)}$ for all $x \in [A,B]$, we
have
$$\gamma \a e^{\a(\xi_x-B)} - z_0'(\xi_x) \geq \gamma \a e^{\a(A-B)} - \max|z_0'| \geq 0 \quad \text{if } \gamma = \a^{-1} e^{\a(B-A)} \max|z_0'|.$$
This implies $(\gamma \a e^{\a(\xi_x-B)} - z_0'(\xi_x))(x-B) \leq 0$ and so $\phi(x) \leq 0$ for all $x \in [A,B]$.

We have thus shown that $\eta \leq z$ on the parabolic boundary and $\tilde{L}(\eta - z) \leq 0$. The parabolic maximum principle thus implies $\eta \leq z$ everywhere. In particular, for any $t \geq 0$, since $x - B \leq 0$, we have 
\begin{align*}
\frac{z(t, x)-z(t, B)}{x-B}\leq  \frac{\eta(t, x)-\eta(t, B)}{x-B}\to \gamma\alpha \quad \text{ as } x \to B,
\end{align*}
giving an upper bound on $z_x(t, B)$.  The same argument with $z$ replaced by $-z$ give a lower bound of $-\gamma\alpha$, in particular 
\begin{align*}
 \lvert z_x(t, B)\rvert\leq \gamma\alpha,\quad \forall t\geq 0.
\end{align*}
The argument works in a similar fashion for the endpoint $x = A$.

\subsubsection{Interior Estimate}

Recall that 
$$L(z) = \left(2G(v_x)-\frac{(z + F)}{\phi} \right)z_{x} + G'(v_x) \phi z.$$
Differentiating this equation w.r.t. $x$ gives
\begin{small}
$$L(z_x) =  \left(2G(v_x)-\frac{(z + F)}{\phi} -\frac{\phi_x}{\phi^2}  \right) z_{xx} + \left(3 G'(v_x) \phi - \frac{(z_x + F')}{\phi} + \frac{(z+F)\phi_x}{\phi^2} \right)z_x + (G'(v_x)\phi)_x z.$$
\end{small}
Consider the auxiliary function
$$\eta = \psi_1(z_x) + \psi_2(z),$$
where $\psi_1, \psi_2$ are functions to be determined. Then as before, we have
\begin{align*}
L(\eta) = & \  \psi_1'(z_x) L(z_x) +  \psi_2'(z)L(z)  - \frac{1}{\phi} \left( \psi_1''(z_x) z_{xx}^2 + \psi_2''(z)z_x^2 \right) \\
=  & \  \psi_1'(z_x) \left[ \left(2G(v_x)-\frac{(z + F)}{\phi} -\frac{\phi_x}{\phi^2}  \right) z_{xx} + \left(3 G'(v_x) \phi - \frac{(z_x + F')}{\phi} + \frac{(z+F)\phi_x}{\phi^2} \right)z_x + (G'(v_x)\phi)_x z\right] \\
& \ +  \psi_2'(z)\left[  \left(2G(v_x)-\frac{(z + F)}{\phi} \right)z_{x} + G'(v_x) \phi z\right]- \frac{1}{\phi} \left( \psi_1''(z_x) z_{xx}^2 + \psi_2''(z)z_x^2 \right).
\end{align*}
Suppose now that $\eta$ attains a maximum value at a point $(t_0,x_0) \in (0,\infty) \times (A,B)$. We have  $\eta_x(t_0,x_0) = 0$ and $L (\eta) \geq 0$ at $(t_0,x_0)$. This implies
$$\psi_1'(z_x) z_{xx} = -\psi_2'(z)z_x \quad \text{at } (t_0,x_0).$$ 
Substituting into the equation for $L(\eta)$ yields
\begin{align*}
0 \leq & \ \left(\frac{ \psi_2'(z) \phi_x}{\phi^2}\right)z_x + \psi_1'(z_x)  \left[ \left(3 G'(v_x) \phi - \frac{(z_x + F')}{\phi} + \frac{(z+F)\phi_x}{\phi^2} \right)z_x + (G'(v_x)\phi)_x z\right] \\
& \ + \psi_2'(z) G'(v_x) \phi z - \frac{1}{\phi} \left( \frac{\psi_1''(z_x) \psi_2'(z)^2}{\psi_1'(z_x)^2} + \psi_2''(z) \right) z_x^2.
\end{align*}
We now choose 
$$\psi_1(s) = \frac{1}{2}\log(s^2), \quad \psi_2(s) = \frac{\a s^2}{2}, \ \a \text{ constant}.$$
Then since
$$\psi_1'(s) = \frac{1}{s}, \quad \psi_1''(s) = -\frac{1}{s^2}, \quad \psi_2'(s) =\a s, \quad \psi_2''(s) = \a,$$
we have
\begin{align*}
0 & \leq \left(\frac{ \a z \phi_x}{\phi^2}\right)z_x + 3 G'(v_x)  - \frac{(z_x + F')}{\phi} + \frac{(z+F)\phi_x}{\phi^2} + (G'(v_x)\phi)_x \left(\frac{z}{z_x}\right) + \a G'(v_x) \phi z^2  - \frac{\a}{\phi} \left( 1 - \a z^2 \right) z_x^2.
\end{align*}
Therefore,
$$\a \left(1-\a z^2 \right) z_x^2 + \left(1 - \frac{ \a z \phi_x}{\phi} \right)z_x \leq \frac{(z+F)\phi_x}{\phi} - F' + 3 \phi G'(v_x) + \phi (G'(v_x)\phi)_x  \left(\frac{z}{z_x}\right) + \a G'(v_x) \phi^2 z^2 .$$
Let $M := \max |w_x|$. Since $|z| \leq M$, if we choose $\a = \frac{1}{2M^2}$, then $1 - \a z^2 \geq 1 - \a M^2 = \frac{1}{2}$. Consequently,
$$\frac{1}{4M^2 }z_x^2 + \left(1 - \frac{ z \phi_x}{2M^2 \phi} \right)z_x \leq \frac{(z+F)\phi_x}{\phi} - F' + 3 \phi G'(v_x) + \phi (G'(v_x)\phi)_x  \left(\frac{z}{z_x}\right) + \frac{G'(v_x) \phi^2 z^2}{2M^2}.$$
Multiplying through by $4M^2$ and then completing the square gives us
$$\left(z_x + 1 - \frac{ z \phi_x}{\phi}\right)^2 \leq 4M^2 \left[\frac{(z+F)\phi_x}{\phi} - F' + 3 \phi G'(v_x) + \phi (G'(v_x)\phi)_x  \left(\frac{z}{z_x}\right) + \frac{G'(v_x) \phi^2 z^2}{2M^2}\right] + \left(1 - \frac{ z \phi_x}{\phi} \right)^2.$$
If $|z_x(t_0,x_0)| \leq 1$, then we have for any $(t,x) \in (0,\infty) \times (A,B)$
$$\eta(t,x) = \frac{1}{2}\log(z_x(t,x)^2) + \frac{z(t,x)^2}{4M^2} \leq  \frac{1}{2}\log(z_x(t_0,x_0)^2) + \frac{z(t_0,x_0)^2}{4M^2} \leq \frac{1}{4}.$$
This, in turn, implies $|z_x(t,x)| \leq e^{\frac{1}{2}}$ for all $(t,x) \in (0,\infty) \times (A,B)$. Therefore, we may assume $|z_x(t_0,x_0)| \geq 1$. This implies
$$\left(z_x + 1 - \frac{ z \phi_x}{\phi}\right)^2 \leq 4M^2 \left[\frac{|(z+F)\phi_x|}{\phi} + |F'| + 3 \phi |G'(v_x)| + \phi |(G'(v_x)\phi)_x | M + \frac{|G'(v_x)| \phi^2}{2} \right] + \left(1 - \frac{ z \phi_x}{\phi} \right)^2.$$
We conclude that
$$|z_x| \leq \bigg|1 - \frac{ z \phi_x}{\phi} \bigg| + \sqrt{4M^2 \left[\frac{|(z+F)\phi_x|}{\phi} + |F'| + 3 \phi |G'(v_x)| + \phi |(G'(v_x)\phi)_x | M + \frac{|G'(v_x)| \phi^2}{2} \right] + \left(1 - \frac{ z \phi_x}{\phi} \right)^2}.$$

\end{document}